\documentclass[11pt,reqno]{amsart}

\usepackage[pdf]{pstricks}
\usepackage{amssymb}
\usepackage{amsmath}
\usepackage{amsfonts}
\usepackage{mathrsfs}
\usepackage{color}
\usepackage{enumerate}
\usepackage[T1]{fontenc}
\usepackage[latin1]{inputenc}
\usepackage{ae,aecompl}
\usepackage{graphicx}
\usepackage{nicefrac}
\usepackage[includehead,includefoot,margin=23mm]{geometry}
\usepackage{bigstrut}
\usepackage{hyperref}
\hypersetup{
    colorlinks=true,       
    linkcolor=blue,          
    citecolor=blue,        
    filecolor=blue,      
    urlcolor=blue           
}

\theoremstyle{plain}
\newtheorem{theorem}{Theorem}[section]
\newtheorem{lemme}[theorem]{Lemma}
\newtheorem{proposition}[theorem]{Proposition}
\newtheorem{corollaire}[theorem]{Corollary}
\newtheorem*{theorem*}{Theorem}
 
\theoremstyle{definition}
\newtheorem{definition}[theorem]{Definition}
\newtheorem{rappel}[theorem]{}

\newtheorem{exemple}[theorem]{Example}
\newtheorem{notation}[theorem]{Notation}

\theoremstyle{remark}
\newtheorem{remarque}[theorem]{Remark}

\newcommand{\PP}[0]{\ensuremath{\mathbb{P}}}

\newcommand{\ZZ}[0]{\ensuremath{\mathbb{Z}}}
\newcommand{\GA}[0]{\ensuremath{\mathbb{G}_{\mathrm{a}}}}
\newcommand{\GM}[0]{\ensuremath{\mathbb{G}_{\mathrm{m}}}}
\newcommand{\AF}[0]{\ensuremath{\mathbb{A}}}

\newcommand{\RR}[0]{\ensuremath{\mathbb{R}}}
\newcommand{\QQ}[0]{\ensuremath{\mathbb{Q}}}
\newcommand{\TT}[0]{\ensuremath{\mathbb{T}}}
\newcommand{\KK}[0]{\ensuremath{\mathbf{k}}}
\newcommand{\OO}[0]{\ensuremath{\mathcal{O}}}

\newcommand{\DD}[0]{\ensuremath{\mathfrak{D}}}

\newcommand{\rt}[0]{\ensuremath{\operatorname{Rt}}}
\newcommand{\chara}[0]{\ensuremath{\operatorname{char}}}

\newcommand{\fract}[0]{\ensuremath{\operatorname{Frac}}}

\newcommand{\spec}[0]{\ensuremath{\operatorname{Spec}}}

\newcommand{\Aut}[0]{\ensuremath{\operatorname{Aut}}}

\newcommand{\divi}[0]{\ensuremath{\operatorname{div}}}

\newcommand{\ord}[0]{\ensuremath{\operatorname{ord}}}
\newcommand{\pol}[0]{\ensuremath{\operatorname{Pol}}}

\newcommand{\rank}[0]{\ensuremath{\operatorname{rank}}}

\newcommand{\homo}[0]{\ensuremath{\operatorname{Hom}}}
\newcommand{\relint}[0]{\ensuremath{\operatorname{rel.int}}}

\def\G{{\mathbb{G}}}

\title[$\GA$-actions on affine $\TT$-varieties of complexity
one]{Additive group actions on affine $\TT$-varieties of
  complexity one in arbitrary characteristic}

\author{Kevin Langlois}
\address{Max Planck Institute for Mathematics, Bonn, Germany}
\email{langlois.kevin18@gmail.com}

\author{Alvaro Liendo} \address{Instituto de Matem\'atica y F\'\i
  sica, Universidad de Talca, Casilla 721, Talca, Chile}
\email{aliendo@inst-mat.utalca.cl}

\date{\today}

\thanks{
  \mbox{\hspace{11pt}}{\it Key words}: torus actions, locally finite
  iterative higher derivations, affine varieties.}

\begin{document}

\maketitle

\begin{abstract}
  Let $X$ be a normal affine $\TT$-variety of complexity at most one
  over a perfect field $\KK$, where $\TT=\GM^n$ stands for the split
  algebraic torus. Our main result is a classification of additive
  group actions on $X$ that are normalized by the $\TT$-action. This
  generalizes the classification given by the second author in the
  particular case where $\KK$ is algebraically closed and of
  characteristic zero.

  With the assumption that the characteristic of $\KK$ is positive, we
  introduce the notion of rationally homogeneous locally finite
  iterative higher derivations which corresponds geometrically to
  additive group actions on affine $\TT$-varieties normalized up to a
  Frobenius map. As a preliminary result, we provide a complete
  description of these $\GA$-actions in the toric situation.
\end{abstract}

\tableofcontents

\section*{Introduction}

Let $\KK$ be an arbitrary field. In this paper a variety $X$ is an
integral separated scheme of finite type over the field $\KK$. We
assume further that $\KK$ is algebraically closed in the field of
rational functions $\KK(X)$. A point in $X$ is a not necessarily
rational closed point. A variety is called normal if all its local
rings are integrally closed domains. All algebraic group actions are,
in particular, regular morphisms.

Let $\TT=\GM^n$ be the $n$-dimensional split algebraic torus, where
$\GM$ stands for the multiplicative group of $\KK$. A $\TT$-variety is
a normal variety endowed with an effective action of $\TT$. The
complexity of a $\TT$-variety $X$ is the non-negative integer $\dim
X-\dim \TT$. If the base field $\KK$ is algebraically closed, then the
complexity of $X$ can be read off geometrically as the codimension of
the generic orbit. The best known examples of $\TT$-varieties are
those of complexity zero, called toric varieties.

Let $\GA$ be the additive group of the field $\KK$. The main result of
this paper is a classification of the $\GA$-actions on an affine
$\TT$-variety $X$ that are normalized by $\TT$ in the cases where $X$
is of complexity zero or one. This generalizes a paper by the second
author \cite{Li}, where the same result is obtained in the particular
case where $\KK$ is algebraically closed and of characteristic
zero. The case of normalized $\GA$-actions on an affine $\GM$-surface
over the field of complex numbers was first studied in \cite{FZ2}.

Let $M$ be the character lattice of $\TT$ and let $N$ be the lattice
of one-parameter subgroups. We have a natural duality $M_{\RR}\times
N_{\RR}\rightarrow \RR$ given by $(m,v)\mapsto\langle m,v\rangle$
between the vector spaces $M_\RR = M\otimes_{\ZZ}\RR$ and $N_\RR =
N\otimes_\ZZ\RR$. Recall that $\TT$-actions on an affine variety
corresponds to $M$-gradings on its coordinate ring.

Affine $\TT$-varieties can be described in combinatorial terms. In the
case of toric varieties, there is the well-known description of affine
toric varieties via strongly convex rational polyhedral cones in
$N_\RR$ \cite{Dem,Od}. In 2006, Altmann and Hausen gave a
combinatorial description of affine $\TT$-varieties of arbitrary
complexity over an algebraically closed field of characteristic zero
\cite{AH}. This intersects with previous works by several authors
\cite{KKMS,Dem88,Ti2,FZ,Ti} (see also \cite{AHS,AIPSV} for the theory
of non-necessarily affine $\TT$-varieties). Furthermore, in a recent
paper, the first author generalized the combinatorial description due
to Altmann and Hausen to the case of affine $\TT$-varieties of
complexity one over an arbitrary field \cite{La2}.

The combinatorial description of affine $\TT$-varieties of complexity
one that we will use in this paper encodes an affine $\TT$-variety $X$
with a triple $(C,\sigma,\DD)$, where $C$ is a regular curve, $\sigma$
is a strongly convex rational polyhedral cone in $N_\RR$ and $\DD$ is
a $\sigma$-polyhedral divisor on $C$, i.e., a divisor in $C$ whose
coefficients instead of integers are polyhedra in $N_\RR$ that can be
decomposed as a Minkowski sum $Q+\sigma$ with $Q$ a compact polyhedron
(see Section~\ref{sec1} for details).

It is well known that the additive group actions on an affine variety
$X=\spec A$ are in one to one correspondence with certain sequences
$\partial = \{\partial^{(i)}:A\rightarrow A\}_{i\in\ZZ_{\geq 0}}$ of
$\KK$-linear operators on $A$ called locally finite iterative higher
derivations \cite{Mi,Cr,CM}, or LFIHDs for short (see
Definition~\ref{sec:sit2.1} for details). Now, assume that $X=\spec A$
is an affine $\TT$-variety and let $\partial$ be an LFIHD on $A$. The
LFIHD $\partial$ is called homogeneous of degree $e\in M$ if every
$\partial^{(i)}$ is homogeneous of degree $ie$. Furthermore, in
positive characteristic, we introduce the technical notion of
rationally homogeneous LFIHDs as follows: let $p>0$ be the
characteristic of $\KK$ and let $r\in \ZZ_{\geq 0}$, then $\partial$
is called rationally homogeneous of degree $e/p^r$ if
$\partial^{(ip^r)}$ is homogeneous of degree $ie$ and
$\partial^{(j)}=0$ whenever $p^r$ does not divide $j$.

In the case where $\KK$ is algebraically closed, the notion of
(rationally) homogeneous LFIHD translates into geometric terms in the
following way. An LFIHD on $A$ is homogeneous if and only if the
corresponding $\GA$-action on $X$ is normalized by the
$\TT$-action. Moreover, let $F_{p^r}:\GA\rightarrow \GA$ be the
Frobenius map sending $t\mapsto t^{p^r}$. If $\partial$ is an LFIHD
and $\phi:\GA\rightarrow \Aut(X)$ is the corresponding $\GA$-action,
then $\partial$ is rationally homogeneous if and only if $\phi\circ
F_{p^r}^{-1}$ is normalized by the $\TT$-action for some $r\in
\ZZ_{\geq 0}$ (see Proposition~\ref{sec:semidirect}). In this case we
say that $\phi$ is normalized by the $\TT$-action up to a Frobenius
map.

The kernel $\ker\partial$ of an LFIHD $\partial$ is defined as the
intersection of $\ker \partial^{(i)}$ for all $i \in \ZZ_{> 0}$; it
is equal to the ring $\KK[X]^{\GA}$ of $\GA$-invariant regular
functions on $X$ and $\fract(\ker\partial)$ corresponds to the field
$\KK(X)^{\GA}$ of $\GA$-invariant rational functions on $X$. Denote by
$\KK(X)^\TT$ the field of $\TT$-invariant rational functions on $X$. A
(rationally) homogeneous LFIHD is called vertical if
$\KK(X)^{\TT}\subseteq \KK(X)^{\GA}$ and horizontal otherwise. When
$\KK$ is algebraically closed, the horizontal condition means
geometrically that the general $\GA$-orbits are transverse to the
rational fibration defined by the $\TT$-action.

Let $X=\spec A$ be the affine toric variety given by the strongly
convex rational cone $\sigma\subseteq N_\RR$. We denote by $\sigma(1)$
the set of extremal rays of the cone $\sigma$. In
Theorem~\ref{sec:th3.5} we classify normalized $\GA$-actions on affine
toric varieties. They are described by Demazure roots of the cone
$\sigma$, i.e., vectors $e\in M$ such that there exists
$\rho\in\sigma(1)$ with $\langle e,\rho\rangle=-1$ and $\langle
e,\rho'\rangle\geq 0$, for all $\rho'\in \sigma(1)$ different from
$\rho$. We also classify $\GA$-actions on affine toric varieties that
are normalized up to a Frobenius map (see Corollary~\ref{cor3.6}). Let
us mention some developments from the theory of Demazure roots.  The
reader may consult \cite{Dem,Cox95,Ni,Ba,Cox14,AHHL} for the study of
automorphisms of complete $\TT$-varieties via Demazure's roots and
\cite{Li3,Ko} for the roots of the affine Cremona groups. See also
\cite{LP} for a geometric description in the setting of affine
spherical varieties.

Let now $X=\spec A$ be an affine $\TT$-variety of complexity one given
by the triple $(C,\sigma,\DD)$. The classification of normalized
$\GA$-actions on such an $X$ is divided into two theorems
corresponding to vertical and horizontal LFIHDs. The classification of
vertical LFIHDs on $A$ is given in Theorem~\ref{sec:th4.4}. They are
described by pairs $(e,\varphi)$, where $e$ is a Demazure root of
$\sigma$ and $\varphi$ is a global section of the invertible sheaf
$\OO_C(\DD(e))$. The $\QQ$-divisor $\DD(e)$ is uniquely determined by
$\DD$ and $e$ in a combinatorial way. The classification of horizontal
LFIHDs on $A$ is only available when $\KK$ is perfect, see
Theorem~\ref{th:5.11}. Its combinatorial counterpart is different from
the characteristic zero case (compare with \cite[Theorem~3.28]{Li})
and is related to the description of rationally homogeneous LFIHDs on
affine toric varieties.

The content of the paper is the following. In Section~\ref{sec1} we
present the combinatorial description of affine $\TT$-varieties of
complexity one that will be used in this paper. In Section~\ref{sec2}
we introduced the background results on $\GA$-actions. In
Section~\ref{sec3} we obtain our classification result for toric
varieties. Finally, the classification of normalized $\GA$-actions on
affine $\TT$-varieties of complexity one is divided in
Sections~\ref{sec4} and \ref{sec5} corresponding to the vertical and
horizontal cases, respectively.

\section{Generalities on affine $\TT$-varieties of complexity one}
\label{sec1}

In this section, we recall a combinatorial description of affine
$\TT$-varieties of complexity one over an arbitrary field
\cite[Section 3]{La2}.  Let $\mathbf{k}$ be field and let $X = \spec
A$ be an affine variety over $\KK$.  We start by introducing some
notation from convex geometry (see e.g. \cite{Od} or \cite[ Section
1]{AH}).

\begin{rappel}
  Let $\TT\simeq \GM^{n}$ be a split algebraic torus over $\KK$.
  Denote by $M = \homo(\TT,\GM)$ the character lattice of $\TT$ and
  let $N = \homo(\GM,\TT)$ be the lattice of one-parameter
  subgroups. We have a natural duality $M_{\RR}\times
  N_{\RR}\rightarrow \RR$ given by $(m,v)\mapsto\langle m,v\rangle$,
  where $M_\RR = M\otimes_{\ZZ}\RR$ and $N_\RR = N\otimes_\ZZ\RR$ are
  the associated real vector spaces. We also let $M_\QQ =
  M\otimes_{\ZZ}\QQ$ and $N_\QQ = N\otimes_\ZZ\QQ$ be the
  corresponding rational vector spaces.

  A \emph{rational cone} in $N_{\RR}$ is a cone generated by a finite
  subset of $N$. If $\sigma\subseteq N_{\RR}$ is a rational cone, then
  we let $\sigma^{\vee}\subseteq M_{\RR}$ be its dual cone, i.e., the
  cone of real linear forms on $M_{\RR}$ that are non-negative on
  $\sigma$. Recall that the dual cone $\sigma^\vee$ of a rational cone
  is again rational. The relative interior of a rational cone
  $\sigma\subseteq N_\RR$, denoted by $\relint(\sigma)$, is the
  topological interior of $\sigma$ in the span of $\sigma$ inside
  $N_\RR$.

  For any face $F\subseteq \sigma$ the set $F^\star$ stands for the
  dual face of $F$ in $\sigma^{\vee}$, i.e., $F^\star=F^\bot\cap
  \sigma^\vee$.  A rational cone $\sigma$ is \emph{strongly convex} if
  $0$ is a face of $\sigma$. This is equivalent to say that the dual
  $\sigma^{\vee}\subseteq M_{\RR}$ is full dimensional. For any
  rational cone $\omega\subseteq M_{\RR}$ we let $\omega_{M} =
  \omega\cap M$.

  Furthermore, given a subsemigroup $S\subseteq M$ we let
  \[
  \KK[S] = \bigoplus_{m\in S}\KK\chi^{m}
  \]
  be the \emph{semigroup algebra} of $S$ defined by the relations
  $\chi^{m}\cdot\chi^{m'} = \chi^{m+m'}$ for all $m,m'\in S$ and
  $\chi^0=1$.

  For any integer $d\geq 0$ and any polyhedron $\Delta\subseteq N_{\RR}$
  we let $\Delta(d)$ be the set of faces of dimension $d$. In
  particular, $\Delta(0)$ is the set of vertices of $\Delta$.

  Let $\sigma\subseteq N_{\RR}$ be a strongly convex rational cone. We
  define $\pol_{\sigma}(N_{\RR})$ as the set of polyhedra in $N_\RR$
  that can be written as a Minkowski sum $Q+\sigma$, where $Q\subseteq
  N_{\RR}$ is a rational polytope, i.e., a bounded polyhedron having
  its vertices in the rational vector space $N_\QQ$.
\end{rappel}

\begin{rappel}
  A \emph{$\TT$-variety} is a normal variety endowed with an effective
  action of the algebraic torus $\TT$.  Recall that a $\TT$-action $X
  = \spec A$ is equivalent to an $M$-grading of the algebra $A$. In
  algebraic terms, a $\TT$-action on $X$ is effective if and only if
  the semigroup of weights of $A$ generates $M$. In this case the
  weight cone $\sigma^\vee$ of $A$ is the dual of a strongly convex
  rational cone $\sigma\subseteq N_{\RR}$.
\end{rappel}

\begin{rappel} \label{sit:1.3} %
  Let $X = \spec A$ be an affine $\TT$-variety. Letting $K_{0}=
  \KK(X)^{\TT}$ be the field of $\TT$-invariant rational functions on
  $X$ we can write
  \[
  A = \bigoplus_{m\in \sigma^{\vee}_{M}}A_{m}\chi^{m}
  \]
  as an $M$-graded subalgebra of $K_{0}[M]$. Here,
  $\sigma^{\vee}\subseteq M_{\RR}$ is the weight cone of $A$,
  $\chi^{m}$ is a weight vector in $\KK(X)$, $A_0=K_0\cap A$, and $A_m$ is
  an $A_0$-module contained in $K_0$. Furthermore, the weight vectors
  satisfy $\chi^0=1$, and $\chi^{m}\cdot\chi^{m'} = \chi^{m+m'}$ for
  all $m,m'\in M$.
 
  The \emph{complexity} of the $\TT$-variety $X$ is the transcendence
  degree of the field extension $K_{0}/\KK$.  Since the action is
  effective, it is also equal to $\rank M - \dim X$. In geometrical
  terms, when $\KK = \bar{\KK}$ is algebraically closed the complexity
  is the codimension of the generic $\TT$-orbit.

  A \emph{toric variety} is a $\TT$-variety of complexity zero. An
  affine toric variety $X=\spec A$ is completely determined by the
  weight cone $\sigma^{\vee}$ of $A$. Conversely, given a strongly
  convex rational cone $\sigma\subseteq N_{\RR}$, we can define an
  affine toric variety by letting $X_{\sigma} :=
  \spec\KK[\sigma^{\vee}_{M}]$.

  Another important class of affine $\TT$-varieties is provided by the
  surface case. If $X$ is an affine $\GM$-surface, then the coordinate
  ring $A = \KK[X]$ is endowed with a $\ZZ$-grading. Up to reversing
  the grading, we can assume that the subspace $A_{+} =
  \bigoplus_{m\in\ZZ_{>0}}A_{m}\chi^{m}$ is nonzero. We distinguish
  three cases (see \cite{FiKa}).

  \begin{enumerate}[$(i)$]
  \item The elliptic case: $A_{-} = \bigoplus_{m\in\ZZ_{<
        0}}A_{m}\chi^{m} = 0$ and $A_{0} = \KK$.
  \item The parabolic case: $A_{-} = 0$ and $A_{0}\neq\KK$.
  \item The hyperbolic case: $A_{-}\neq 0$.
  \end{enumerate}

  More generally, an affine $\TT$-variety $X = \spec A$ of complexity
  one is called \emph{elliptic} if $A_{0} = \KK$ (see \cite[Section
  1.1]{Li}).
\end{rappel}

To provide a description of affine $\TT$-varieties of complexity one,
we need to consider the Weil divisors theory on regular algebraic
curves. In the next paragraph, we recall the definitions we need.

\begin{rappel}
  Let $C$ be a regular curve over $\KK$. By a point belonging to $C$
  we mean a closed point. Letting $z\in C$ we let $[\kappa_{z}:\KK]$
  be the \emph{degree} of the point $z$ defined as the dimension of
  residue field $\kappa_{z}$ of $z$ over $\KK$ (see \cite[Proposition
  1.1.15]{St}). A point $z\in C$ of degree one is called a
  \emph{rational point}. \rm For a nonzero rational function
  $f\in\KK(C)^*$ the associated principal divisor is
  \[
  \divi f = \sum_{z\in C} \ord_{z}f \cdot z\,,
  \]
  where $\ord_{z}f$ is the order of $f$ at the point $z$. The
  \emph{degree} of a Weil $\QQ$-divisor $D = \sum_{z\in C}a_z\cdot z$
  is the rational number
  \[
  \deg D = \sum_{z\in C}[\kappa_{z} : \KK]\cdot a_{z}\,.
  \]
  If $C$ is projective, then we have $\deg\divi f=0$ (see \cite[Theorem
  1.4.11]{St}).  In addition, we let $\lfloor D\rfloor = \sum_{z\in
    C}\lfloor a_{z}\rfloor\cdot z$ be the integral Weil divisor
  obtained by taking the integral part of each coefficient of
  $D$. Similary, the $\QQ$-divisor $\{ D\} = D - \lfloor D\rfloor$
  stands for the fractional part of $D$. The space of global sections
  of the $\QQ$-divisor $D$ is defined by
  \[
  H^0(C,\OO_C( D)):=H^0(C,\OO_C(\lfloor D\rfloor)) =
  \left\{f\in\KK(C)^*\,|\,\divi f + D \geq 0\right\}\cup\{0\}.
  \]
  When $C$ is projective, $H^0(C,\OO_C( D))$ is usually called the
  \emph{Riemann-Roch space} of $D$.
\end{rappel}

The following has been introduced in \cite{AH} for any complexity in
the case where $\KK$ is algebraically closed of characteristic
zero. In our context, we give a similar definition.

\begin{definition} \label{sec:def-convex} %
  Let $C$ be a regular curve over $\KK$. Consider $\sigma\subseteq
  N_{\RR}$ a strongly convex rational cone. A
  \emph{$\sigma$-polyhedral divisor} over $C$ is a formal sum $\DD =
  \sum_{z\in C}\Delta_{z}\cdot z$, where each
  $\Delta_{z}\in\pol_{\sigma}(N_{\RR})$ and $\Delta_{z} = \sigma$ for
  all but finitely number of $z$.  For every coefficient $\Delta_{z}$
  of the $\sigma$-polyhedral divisor $\DD$ we define $h_z$ as the
  piecewise linear map $h_z:M_{\RR}\rightarrow \RR$ given by $m\mapsto
  \min_{v\in\Delta_{z}(0)}\langle m, v\rangle$. We remark that $h_z$
  restricted to $\sigma^\vee\subseteq M_\RR$ corresponds to the
  support function of $\Delta_z$.

  For any $m\in M_{\QQ}$ we define the \emph{evaluation} of $\DD$ as
  the $\QQ$-divisor
  \[
  \DD(m) = \sum_{z\in C}h_{z}(m)\cdot z\,.
  \]
  We denote by $\Lambda(\DD)$ the coarsest refinement of the quasifan
  of $\sigma^{\vee}$ such that the map $m\mapsto\DD(m)$ is linear in
  each cone. We also define the \emph{degree} of $\DD$ as
  \[
  \deg\DD = \sum_{z\in
    C}[\kappa_{z}:\KK]\cdot\Delta_z\in\pol_\sigma(N_\RR)\,.
  \]
 
  A $\sigma$-polyhedral divisor $\DD=\sum_{z\in C}\Delta_{z}\cdot z$
  is called \emph{proper} if it satisfies one of the following
  conditions.
  \begin{enumerate}[$(i)$]
  \item[\rm (i)] the curve $C$ is affine, or
  \item[\rm (ii)] the curve $C$ is projective, the polyhedron
    $\deg\DD$ is a proper subset of $\sigma$, and for every
    $m\in\sigma^{\vee}_M$ such that $\deg\DD(m)=0$, a nonzero integral
    multiple of $\DD(m)$ is principal.
  \end{enumerate}
\end{definition}

Actually, polyhedral divisors are combinatorial objects that allow us
to construct multigraded algebras, as explained in the following.

\begin{notation}
  To a $\sigma$-polyhedral divisor $\DD= \sum_{z\in C}\Delta_{z}\cdot
  z$ over $C$ we associate the rational $\TT$-submodule
  \[
  A[C,\DD] = \bigoplus_{m\in\sigma^{\vee}_{M}}A_m\cdot\chi^{m}\subseteq
  K_{0}[M],\quad\mbox{where}\quad A_m=H^{0}\big(C,
  \OO_{C}(\DD(m))\big)\mbox{ and } K_{0} = \KK(C)\,.
  \]
  Given $m,m'\in \sigma^\vee_M$, the evaluations satisfy
  $\DD(m)+\DD(m')\leq \DD(m+m')$. Hence, for every $f\in A_m$ and
  every $g\in A_{m'}$, the product $fg$ lies on $ A_{m+m'}$. This
  multiplication rule turns the vector space $A[C,\DD]$ into an
  $M$-graded subalgebra.

  For a non-empty open subset $C_{0}\subseteq C$ we let
  \[
  \DD_{|C_{0}} =\sum_{z\in C_{0}}\Delta_{z}\cdot z
  \]
  be the \emph{restriction} of $\DD$ to $C_{0}$.
\end{notation}

The following yields a description of the coordinate ring of an affine
$\TT$-variety of complexity one (for a proof see \cite[Theorem
4.3]{La2}). This description intersects with some classical cases; see
\cite{Ti}, \cite{Ti2} for complexity one case, \cite{AH} for higher
complexity, and \cite{FZ} for the Dolgachev-Pinkham-Demazure
presentation of affine complex $\mathbb{C}^*$-surfaces. For the
functorial properties of this description see
\cite[Proposition~4.5]{La2}.

\begin{theorem}
  \begin{enumerate}[$(i)$]
  \item If $\mathfrak{D}$ is a proper $\sigma$-polyhedral divisor on a
    regular curve $C$ over $\mathbf{k}$, then the $M$-graded algebra
    $A[C,\mathfrak{D}] = \bigoplus_{m\in\sigma^{\vee}\cap M}A_{m},$
    where $$A_{m} = H^{0}(C,\mathcal{O}_{C}(\mathfrak{D}(m))),$$ is
    the coordinate ring of an affine $\mathbb{T}$-variety of
    complexity one over $\mathbf{k}$.
  
  \item Conversely, to any affine $\mathbb{T}$-variety $X = \rm
    Spec\,\it A$ of complexity one over $\mathbf{k}$, one can
    associate a pair $(C_{X}, \mathfrak{D}_{X,\gamma})$ as follows.
  
    \begin{enumerate}[(a)]
    \item $C_{X}$ is the abstract regular curve over $\mathbf{k}$
      defined by the conditions $\mathbf{k}[C_{X}] =
      \mathbf{k}[X]^{\mathbb{T}}$ and $k(C_{X}) = k(X)^{\mathbb{T}}$.

    \item $\mathfrak{D}_{X,\gamma}$ is a proper
      $\sigma_{X}$-polyhedral divisor over $C_{X}$, which is uniquely
      determined by $X$ and by a sequence $\gamma = (\chi^{m})_{m\in
        M}$ of $k(X)$ as in~\ref{sit:1.3}.
    \end{enumerate}

    We have a natural identification $A =
    A[C_{X},\mathfrak{D}_{X,\gamma}]$ of $M$-graded algebras with the
    property that every homogeneous element $f\in A$ of degree $m$ is
    equal to $f_{m}\chi^{m}$, for a unique global section $f_{m}$ of
    the sheaf $\mathcal{O}_{C_{X}}(\mathfrak{D}_{X,\gamma}(m))$.
  \end{enumerate}
\end{theorem}

\begin{exemple} \label{ex:1.8} Let $M=\ZZ^2$ and let $\sigma$ be the
  first quadrant in the vector space $N_\RR=\RR^2$. We also let
  $\Delta_0=(1/2,0)+\sigma$, $\Delta_1=L+\sigma$ and
  $\Delta_\infty=(1/2,0)+\sigma$, where $L$ is the line segment
  joining the points $(0,0)$ and $(-1/2,1/2)$.

  
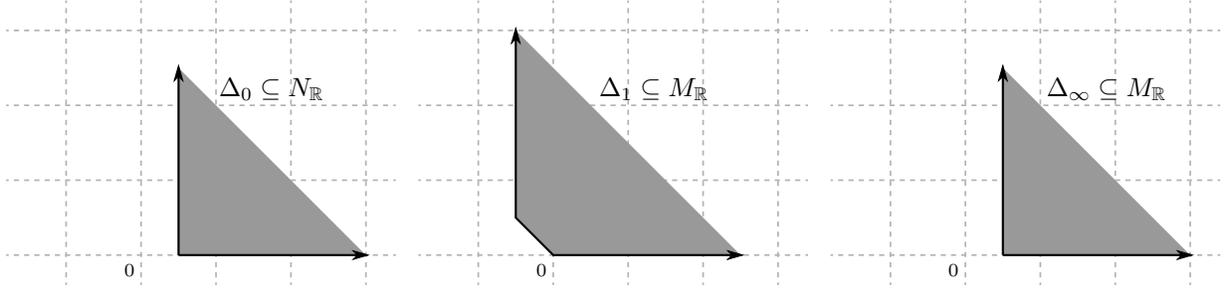
\begin{figure}[!ht]
    \input{cplone.tex}
  \end{figure}
  
  Letting $\KK$ be an arbitrary field and $C=\PP^1_\KK$ we let $\DD$ be
  the $\sigma$-polyhedral divisor
  $\DD=\Delta_0\cdot[0]+\Delta_1\cdot[1]+\Delta_\infty\cdot[\infty]$
  over $C$. The degree of $\DD$ is $\deg\DD=L'+\sigma$, where $L'$ is
  the line segment joining the points $(1,0)$ and $(1/2,1/2)$.

  
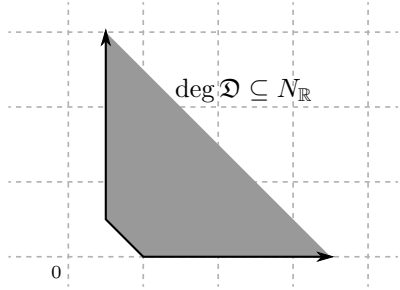
\begin{figure}[!ht]
\psset{xunit=.8pt,yunit=.8pt,runit=.8pt}
\begin{pspicture}(184.25,134.625)
{
\newrgbcolor{curcolor}{0.7019608 0.7019608 0.7019608}
\pscustom[linewidth=0.80000001,linecolor=curcolor,linestyle=dashed,dash=2.4 2.4]
{
\newpath
\moveto(184.26305,120.453524)
\lineto(0.01109,120.453524)
}
}
{
\newrgbcolor{curcolor}{0.7019608 0.7019608 0.7019608}
\pscustom[linewidth=0.80000001,linecolor=curcolor,linestyle=dashed,dash=2.4 2.4]
{
\newpath
\moveto(184.26305,85.02045)
\lineto(0.01108,85.02045)
}
}
{
\newrgbcolor{curcolor}{0.7019608 0.7019608 0.7019608}
\pscustom[linewidth=0.80000001,linecolor=curcolor,linestyle=dashed,dash=2.4 2.4]
{
\newpath
\moveto(184.26305,49.58738)
\lineto(0.01108,49.58738)
}
}
{
\newrgbcolor{curcolor}{0.7019608 0.7019608 0.7019608}
\pscustom[linewidth=0.80000001,linecolor=curcolor,linestyle=dashed,dash=2.4 2.4]
{
\newpath
\moveto(184.26305,14.15431)
\lineto(0.01107,14.15431)
}
}
{
\newrgbcolor{curcolor}{0.7019608 0.7019608 0.7019608}
\pscustom[linewidth=0.80000001,linecolor=curcolor,linestyle=dashed,dash=2.4 2.4]
{
\newpath
\moveto(28.35753,134.626752)
\lineto(28.35753,-0.0121)
}
}
{
\newrgbcolor{curcolor}{0.7019608 0.7019608 0.7019608}
\pscustom[linewidth=0.80000001,linecolor=curcolor,linestyle=dashed,dash=2.4 2.4]
{
\newpath
\moveto(63.79061,134.626752)
\lineto(63.79061,-0.0121)
}
}
{
\newrgbcolor{curcolor}{0.7019608 0.7019608 0.7019608}
\pscustom[linewidth=0.80000001,linecolor=curcolor,linestyle=dashed,dash=2.4 2.4]
{
\newpath
\moveto(99.22368,134.626752)
\lineto(99.22368,-0.0121)
}
}
{
\newrgbcolor{curcolor}{0.7019608 0.7019608 0.7019608}
\pscustom[linewidth=0.80000001,linecolor=curcolor,linestyle=dashed,dash=2.4 2.4]
{
\newpath
\moveto(134.65675,134.626752)
\lineto(134.65675,-0.0121)
}
}
{
\newrgbcolor{curcolor}{0.7019608 0.7019608 0.7019608}
\pscustom[linewidth=0.80000001,linecolor=curcolor,linestyle=dashed,dash=2.4 2.4]
{
\newpath
\moveto(170.08982,134.626752)
\lineto(170.08982,-0.01209)
}
}
{
\newrgbcolor{curcolor}{0.60000002 0.60000002 0.60000002}
\pscustom[linestyle=none,fillstyle=solid,fillcolor=curcolor]
{
\newpath
\moveto(152.37328,14.15432)
\lineto(63.7906,14.15432)
\lineto(46.07407,31.87085)
\lineto(46.07407,120.453523)
\closepath
}
}
{
\newrgbcolor{curcolor}{0 0 0}
\pscustom[linewidth=1,linecolor=curcolor]
{
\newpath
\moveto(46.07407,120.453523)
\lineto(46.07407,31.87085)
\lineto(63.7906,14.15432)
\lineto(152.37328,14.15432)
}
}
{
\newrgbcolor{curcolor}{0 0 0}
\pscustom[linestyle=none,fillstyle=solid,fillcolor=curcolor]
{
\newpath
\moveto(46.07407,116.453523)
\lineto(48.07407,114.453523)
\lineto(46.07407,121.453523)
\lineto(44.07407,114.453523)
\lineto(46.07407,116.453523)
\closepath
}
}
{
\newrgbcolor{curcolor}{0 0 0}
\pscustom[linewidth=0.5,linecolor=curcolor]
{
\newpath
\moveto(46.07407,116.453523)
\lineto(48.07407,114.453523)
\lineto(46.07407,121.453523)
\lineto(44.07407,114.453523)
\lineto(46.07407,116.453523)
\closepath
}
}
{
\newrgbcolor{curcolor}{0 0 0}
\pscustom[linestyle=none,fillstyle=solid,fillcolor=curcolor]
{
\newpath
\moveto(148.37328,14.15432)
\lineto(146.37328,12.15432)
\lineto(153.37328,14.15432)
\lineto(146.37328,16.15432)
\lineto(148.37328,14.15432)
\closepath
}
}
{
\newrgbcolor{curcolor}{0 0 0}
\pscustom[linewidth=0.5,linecolor=curcolor]
{
\newpath
\moveto(148.37328,14.15432)
\lineto(146.37328,12.15432)
\lineto(153.37328,14.15432)
\lineto(146.37328,16.15432)
\lineto(148.37328,14.15432)
\closepath
}
}

\rput(22.77093506,7.06771851){\tiny{$0$}}
\rput(111.35358923,93){\small{$\deg\DD\subseteq N_\RR$}}

\end{pspicture}
  \end{figure}
 
  Hence $\deg\DD\subsetneq \sigma$ and $\DD$ is proper. Let
  $A=A[C,\DD]$ and $X=\spec A$. A direct computation shows that the
  elements
  \[u_1=\frac{t-1}{t}\cdot\chi^{(2,0)},\quad u_2=\chi^{(0,1)},\quad
  u_3=\chi^{(1,1)},\quad u_4=\frac{(t-1)^2}{t}\cdot\chi^{(2,0)},
  \quad\mbox{and}\quad u_5=\frac{(t-1)^2}{t}\cdot\chi^{(3,0)}\]
  generate the algebra $A$. Furthermore, a minimal set of relations
  satisfied by these generators is given by $u_2u_5-u_3u_4=0$,
  $u_3u_5-u_1^2u_2-u_1u_2u_4=0$ and $u_5^2-u_1^2u_4-u_1u_4^2=0$. Hence
  \[A\simeq
  k[x_1,x_2,x_3,x_4,x_5]/(x_2x_5-x_3x_4\, ,\, x_3x_5-x_1^2x_2-x_1x_2x_4\, ,\,
  x_5^2-x_1^2x_4-x_1x_4^2)\,.\]
\end{exemple}

The following result provides a calculation of the Altmann--Hausen
presentation in terms of polyhedral divisors when we extend the
scalars to an algebraic closure of $\KK$, see \cite[Proposition
3.9]{La2}.

\begin{lemme}\label{lem1.8}
  Assume that $\mathbf{k}$ is a perfect field, and let
  $\bar{\mathbf{k}}$ be an algebraic closure of $\mathbf{k}$.  The
  absolut Galois group of $\mathfrak{G}_{\bar{\mathbf{k}}/\mathbf{k}}$
  acts on the closed points of the curve $$C_{\bar{\mathbf{k}}} =
  C\times_{{\rm Spec}\, \mathbf{k}}{\rm Spec}\,\bar{\mathbf{k}}$$
  which can be identified with the set of the
  $\bar{\mathbf{k}}$-rational points of $C(\bar{\mathbf{k}})$.  The
  orbit space
  $C(\bar{\mathbf{k}})/\mathfrak{G}_{\bar{\mathbf{k}}/\mathbf{k}}$ can
  be identified with $C$. We denote by $S :
  C(\bar{\mathbf{k}})\rightarrow C$ the quotient map.  If
  $\mathfrak{D} = \sum_{z\in C}\Delta_{z}\cdot z$ is a proper
  $\sigma$-polyhedral divisor over $C$, then
  \[
  A[C,\mathfrak{D}]\otimes_{\mathbf{k}}\bar{\mathbf{k}} =
  A\left[C(\bar{\mathbf{k}}), \mathfrak{D}_{\bar{\mathbf{k}}}\right],
  \]
  where $\mathfrak{D}_{\bar{\mathbf{k}}}$ is the proper
  $\sigma$-polyhedral divisor over $C(\bar{\mathbf{k}})$ defined by
  \[
  \mathfrak{D}_{\bar{\mathbf{k}}} = \sum_{z\in C}\Delta_{z}\cdot
  S^{\star}(z)\,\,\,\rm with\,\,\,\it S^{\star}(z) = \sum_{z'\in
    S^{-\rm 1\it}(z)}z'.
  \]
\end{lemme}

The proof of the following result is exactly the same as in
\cite[Lemma~1.6]{Li}.

\begin{lemme}\label{lm1.9}
  Let $A = A[C,\DD]$, where $C$ is a regular curve over $\KK$ with
  field of rational functions $K_{0}$ and $\DD = \sum_{z\in
    C}\Delta_{z}\cdot z$ is a proper $\sigma$-polyhedral
  divisor. Consider the normalization $A'$ of the cyclic extension
  $A[s\chi^{e}]$, where $e\in M$, $s^{d}\in A$ homogeneous of degree
  $de$, and $d\in\ZZ_{>0}$. If $\KK$ is algebraically closed in $A'$,
  then $A' = A[C',\DD']$ where $C'$ and $\mathfrak{D'}$ are defined by
  the following.
  \begin{enumerate}[$(i)$]
  \item If $A$ is elliptic, then $A'$ is also and $C'$ is the regular
    projective curve associated with the algebraic function field
    $K_{0}[s]$.

  \item If $A$ is non-elliptic, then $A'$ is also and $C' = \rm
    Spec\,\it A'_{\rm 0}$, where $A'_{0}$ is the normalization of
    $A_{0}$ in $K_{0}[s]$.

  \item In both cases $\DD' = \sum_{z\in C}\Delta_{z}\cdot
    \pi^{*}(z)$, where $\pi:C'\rightarrow C$ is the natural
    projection.
  \end{enumerate}
\end{lemme}

\section{Generalities on $\GA$-actions} \label{sec2}

Let $X=\spec A$ be an affine $\TT$-variety over an arbitrary field
$\KK$. In this section, we study the relation between $\GA$-actions
on $X$ that are normalized by the torus action and homogeneous locally
finite iterative higher derivations.

\begin{definition} \label{sec:sit2.1} Let $\partial =
  \{\partial^{(i)}\}_{i\in\ZZ_{\geq 0}}$ be a sequence of $\KK$-linear
  operators on $A$. We say that $\partial$ is a \emph{locally finite
    iterative higher derivation} (LFIHD for short) if it satisfies the
  following conditions:
  \begin{enumerate}[$(i)$]

  \item The operator $\partial^{(0)}$ is the identity map.

  \item For any $i\in\ZZ_{\geq 0}$ and for all $f_{1},f_{2}\in A$ we have the
    \em Leibniz rule \rm
    \[
    \partial^{(i)}(f_{1}\cdot f_{2}) = \sum_{j =
      0}^{i}\partial^{(j)}(f_{1})\cdot
    \partial^{(i-j)}(f_{2})\,.
    \]

  \item The sequence $\partial$ is locally finite, i.e. for any $f\in
    A$ there exists a positive integer $r$ such that for any $i\geq
    r$, $\partial^{(i)}(f) = 0$.

  \item For all $i,j\in\ZZ_{\geq 0}$ and for any regular function $f\in A$ we
    have
    \[
    \left(\partial^{(i)}\circ\partial^{(j)}\right) (f) =
    \binom{i+j}{i}\,\partial^{(i+j)}(f)\,.
    \]
  \end{enumerate}
  Furthermore, if $\partial$ verifies only $(i), (ii), (iv)$, we say
  that $\partial$ is a \emph{iterative higher derivation}. If
  $\partial$ verifies only $(i),(ii)$, we say $\partial$ is a
  \emph{Hasse-Schmidt derivation} (see \cite{Vo}).
\end{definition}

Consider an action
\[\phi:\GA\times X\rightarrow X\]
of the additive group $\GA$ over $\KK$.  Then the comorphism $\phi^*$
gives a sequence $\partial = \{\partial^{(i)}\}_{i\in\ZZ_{\geq 0}}$ of
$\KK$-linear operators on $A$ defined by the following way. For any
$f\in A$ we write
\[
\phi^{*}(f) = \sum_{i = 0}^{\infty}\partial^{(i)}(f)\cdot x^{i}\in
A\otimes_{\KK}\KK[x],\quad\mbox{where}\quad \KK[x] = \KK[\GA]
\]
is the polynomial algebra in one variable. An easy computation shows
that $\partial$ is an LFIHD \cite{Mi}. Conversely, given an LFIHD
$\partial$ on $A$, its \emph{exponential map}
\[
e^{x\partial} := \sum_{i = 0}^{\infty}\partial^{(i)}\,x^{i}
\]
is the comorphism of a $\GA$-action on $X = \spec A$.

\begin{remarque}
  Consider an LFIHD $\partial$ on $A$. For a positive integer $i$ we
  let
  \[
  \left(\partial^{(1)}\right )^{\circ \,i}
  = \partial^{(1)}\circ\ldots\circ\partial^{(1)}
  \]
  be the composition of $i$ copies of $\partial^{(1)}$. Denoting by
  $p$ the characteristic of the field $\KK$, we have the equality
  \[
  \partial^{(i)} =
  \frac{\left(\partial^{(1)}\right)^{\circ\,i_{0}}\circ
    \left(\partial^{(p)}\right)^{\circ\,i_{1}}\circ\ldots \circ
    \left(\partial^{(p^{r})}\right)^{\circ\,i_{r}}}
  {(i_{0})!(i_{1})!\ldots (i_{r})!}\,,
  \]
  where $i = \sum_{j = 0}^{r}i_{j}\cdot p^{j}$ is the $p$-adic
  expansion\footnote{ When $p = 0$ we make the convention that the
    $p$-adic expansion is $i = i_{0}$.} of $i$. If further $p = 0$,
  then the $\GA$-action is therefore uniquely determined by the
  locally nilpotent derivation $\partial^{(1)}$.
\end{remarque}

In characteristic zero, the algebra of invariants of a $\GA$-action on
the variety $X = \spec A$ is the kernel of the associated locally
nilpotent derivation on $A$.  The following definition describes the
arbitrary characteristic case.

\begin{definition}
  For an LFIHD $\partial$ on the algebra $A$ its \emph{kernel} is the
  subset
  \[
  \ker\partial := \left\{\, f\in A\mid \partial^{(i)}(f) = 0, \mbox{
      for all } i\in\ZZ_{>0}\right\}.
  \]
  This is the subalgebra of invariants $A^{\GA}\subseteq A$ for the
  $\GA$-action corresponding to $\partial$. The LFIHD $\partial$ is
  \emph{non-trivial} if $\ker \partial \neq A$.  A subspace
  $V\subseteq A$ is called \emph{$\partial$-invariant} if for any
  $i\in\ZZ_{\geq 0}$, we have the inclusion
  $\partial^{(i)}(V)\subseteq V$. In particular, the subspace
  $\ker\partial$ is $\partial$-invariant. For any $f\in A$ we define
  the multiplication $f\partial$ as the sequence of $\KK$-linear
  operators $f\partial=\{f^i\partial^{(i)}\}_{i\in \ZZ_{\geq 0}}$. It
  is easy to check that $f\partial$ is an LFIHD if and only if $f\in
  \ker\partial$.
\end{definition}

The next result provides some useful properties of $\GA$-actions, see
\cite[2.1, 2.2]{CM} and \cite[Example 3.5]{Cr}. 

\begin{proposition}\label{sec:LFIHD}
  For every non-trivial LFIHD $\partial$ on the algebra $A$ the
  following hold.

  \begin{enumerate}[$(a)$]

  \item The subring $\ker\partial\subseteq A$ is factorially
    closed, i.e., for all $f_{1},f_{2}\in A$ we have $f_{1}f_{2}\in
    \ker\partial \setminus\{0\}$ implies $f_{1},
    f_{2}\in\ker\partial$.
    
  \item The subring $\ker\partial$ is algebraically closed in $A$.

  \item The subring $\ker\partial$ is a subring of codimension one in
    $A$.

  \item If $\chara(\KK) = p>0$ and $A=\KK[y]$ is the polynomial ring
    in one variable, then there are some $c_{1},\ldots, c_{r}\in
    \KK^{*}$ and some integers $0\leq s_{1}<\ldots<s_{r}$ such that
    \[
    e^{x\partial}(y) = y + \sum_{i = 1}^{r}c_{i}\cdot x^{p^{s_{i}}}.
    \]
  \item If $A^{*}$ is the set of units of $A$, then $A^{*}\subseteq
    \ker\partial$ so that $A^{*} = \left(\ker\partial\right)^{*}$.

  \item A principal ideal $(f) = fA$ is $\partial$-invariant if and
    only if $f\in\ker\partial$.

  \end{enumerate}
\end{proposition}

\begin{proof}
  Assertions $(a), (b)$ and $(c)$ are obtained by using the degree
  function
  \[
  A\setminus \{0\}\rightarrow \ZZ_{\geq 0},\,\,\,f\mapsto
  \deg_{x}e^{x\partial}(f)\,.
  \]
  In particular, we remark that $(b)$ implies that the ring
  $\ker\partial$ is normal whenever $A$ is normal. Assertion $(d)$ is
  proven in \cite[Example 3.5]{Cr}. Assertion $(e)$ is an easy
  consequence of $(a)$.

  Using arguments from \cite[$2$, $1.2\,\it (b)$]{FZ} we give a short
  proof of $(f)$.  Assume that $f$ is nonzero. By
  Definition~\ref{sec:sit2.1}~$(iii)$ we can consider $d\in\ZZ_{\geq0}$
  such that $f':= \partial^{(d)}(f)\neq 0$ and belongs to
  $\ker\partial$.  If the ideal $(f)$ is $\partial$-invariant, then
  $f'\in\ker\partial\cap (f)$ so that $f' = af$ for some $a\in A$.  By
  Proposition~\ref{sec:LFIHD}~$(a)$ we obtain $f\in\ker\partial$.
  Conversely, let $a'\in A$. By Definition~\ref{sec:sit2.1}~$(ii)$,
  for any $i\in\ZZ_{\geq 0}$ we have $\partial^{(i)}(a'f)
  = \partial^{(i)}(a')f$ and so the ideal $(f)$ is
  $\partial$-invariant.
\end{proof}

In the next lemma, we study the extensions of LFIHDs on the algebra
$A$ to the localization ring $T^{-1}A$ given by a multiplicative
system $T\subseteq A$. We were inspired by well-known computations
with the Hasse-Teichm\"uller derivatives
(cf. \cite[Section~2]{JKS}). For this lemma, we let
\[
E(i,j) = \left\{(s_{1},\ldots, s_{j})\in\ZZ^{j}_{>0}\mid \sum_{\ell =
    1}^{j}s_{\ell} = i\right\} \quad \mbox{for all integers }
i,j\in\ZZ_{>0},\mbox{ such that }j\leq i\,.
\]

\begin{lemme} \label{sec:mult-syst}
  Let $T$ be a subset of $A$ stable under multiplication such that
  $0\not\in T$ and $1\in T$.
  \begin{enumerate}[$(i)$]
  \item If $\partial$ be an iterative higher derivation on the algebra
    $A$, then $\partial$ extends to a unique iterative higher
    derivation $\bar{\partial} = \{\bar{\partial}^{(i)}\}_{i\in\ZZ_{\geq 0}}$
    on the algebra $T^{-1}A$ given by
    \[
    \bar{\partial}^{(i)}\left(\frac{1}{f}\right) = \sum_{j =
      1}^{i}\frac{(-1)^{j}}{f^{j+1}} \sum_{(s_{1},\ldots,s_{j})\in
      E(i,j)}
    \partial^{(s_{1})}(f)\ldots\partial^{(s_{j})}(f)
    \]
    for all $f\in T$ and all $i\in\ZZ_{>0}$.

  \item Furthermore, if $\partial$ is an LFIHD on $A$ and if $T\subseteq
    \ker\partial$, then the extension $\bar{\partial}$ on $T^{-1}A$ is
    an LFIHD.
  \end{enumerate}
\end{lemme}

\begin{proof}
  The existence and the uniqueness of $\bar{\partial}$ is given in
  \cite[$3.7$, $5.8$]{Ma}, \cite[Section~3]{Vo}. Proceeding by
  induction the computation of $\bar{\partial}^{(i)}(\frac{1}{f})$ is
  an easy consequence of Definition~\ref{sec:sit2.1}~$(ii)$. The rest
  of the proof is straightforward.
\end{proof}

As a consequence of the previous lemma, we obtain a result on
equivariant cyclic coverings of an affine variety with a $\GA$-action
(see also \cite[Lemma $1.8$]{FZ2}).

\begin{corollaire} \label{cor:2.6}
  Let $K = \fract A$. Consider an LFIHD $\partial$ on $A$ and let
  $f\in \ker\partial$ be a nonzero element. Let $d\in\ZZ_{>0}$ be an
  integer and let $u$ be an algebraic element over $K$ satisfying
  $u^{d}-f = 0$.  If $B$ is the integral closure of $A[u]$ in its
  field of fractions, then $\partial$ extends to a unique LFIHD
  $\partial'$ on the algebra $B$ such that $u\in\ker\partial'$.
\end{corollaire}
\begin{proof}
  By Lemma~\ref{sec:mult-syst} we can extend the LFIHD $\partial$ on
  $A$ to an iterative higher derivation on the field $K$, and on the
  polynomial ring $K[t]$ by letting $\bar{\partial}^{(i)}(t) = 0$ for
  any $i\geq 1$. Consider the morphism of $K$-algebras
  $\phi:K[t]\rightarrow K[u]$, $t\mapsto u$. Let $P\in K[t]$ be the
  monic polynomial generating the ideal $\ker\phi$.

  We can write $t^{d}-f = FP$, for some $F\in K[t]$.  Remark that $F$
  is monic since $P$ and $t^d-f$ are monic.  Since $A$ is integrally
  closed, we obtain $F,P\in A[t]$. Furthermore, for any $i\in\ZZ_{>0}$
  we have $\bar{\partial}^{(i)}(FP) = \bar{\partial}^{(i)}(t^{d}-f) =
  0$. Note that $A[t]$ is $\bar{\partial}$-invariant and the
  restriction of $\bar{\partial}$ to $A[t]$ is an LFIHD. Therefore, by
  Proposition~\ref{sec:LFIHD}~$(a)$, we have $P\in A[t]\cap \ker
  \bar{\partial}$ defining an iterative higher derivation $\partial'$
  on $K[u]$. Clearly, the normalization $B$ of the ring $A[u]$ is
  again $\partial'$-invariant. The rest of the proof is
  straightforward and we omitted it.
\end{proof} 

In the sequel, we let
\[
A=\bigoplus_{m\in\sigma^{\vee}_{M}}A_{m}\chi^{m}\subseteq K_{0}[M]
\]
as in Section~\ref{sec1}, where $\chi^m$ is also seen as the character
of the split torus $\TT$ corresponding to the lattice vector $m\in
M$. Let us introduce the notion of homogeneous iterative higher
derivations.
\begin{definition}
  Let $\partial$ be an iterative higher derivation. The sequence
  $\partial$ is \emph{homogeneous} if there exists $e\in M$ such that
  \[
  \partial^{(i)}(A_{m}\chi^{m})\subseteq
  A_{m+ie}\chi^{m+ie}\quad\mbox{for all}\quad i\in\ZZ_{\geq 0} \mbox{ and }
  m\in M\,.
  \]
  If $\partial$ is non-trivial, then the vector $e$ is called the \em
  degree \rm of $\partial$ and is denoted by $\rm
  deg\,\it \partial$. For the case where $\KK$ is of characteristic
  $p>0$ we have the more general definition.  Given $r\in\ZZ_{\geq 0}$
  we say that $\partial$ is \emph{rationally homogeneous} of degree
  $e/p^{r}$ (or of bidegree $(e,p^r)$ if we need to emphasize the
  vector $e$) if it satisfies the following.
  \begin{enumerate}[$(i)$]

  \item $\partial^{(ip^r)}(A_{m}\chi^{m})\subseteq A_{m+ie}\chi^{m+ie}$,
    for all $i\in\ZZ_{\geq 0}$, and $m\in M$.

  \item $\partial^{(j)} = 0$ whenever $p^{r}$ does not divide $j$.
  \end{enumerate}
\end{definition}

In \cite[Section~1.2]{Li} it is shown that a usual derivation on a
multigraded algebra which sends graded pieces into graded pieces is
homogeneous.  However this does not hold for higher derivations.  Note
also that the kernel of a homogeneous LFIHD $\partial$ on $A$ is an
$M$-graded subalgebra of $A$.  In the sequel, we introduce some
notation in order to have a geometrical interpretation of homogeneous
and rationally homogeneous LFIHDs in the case where $\KK$ is an
algebraically closed field\footnote{Note that the
  Notation~\ref{sec:semidirect} and Proposition~\ref{sec:norm-graded}
  can be generalized in the setting of group schemes and of Hopf
  algebras when $\KK$ is arbitrary.}.

\begin{notation} \label{sec:semidirect} Assume that $\KK$ is
  algebraically closed.  Letting $e\in M$ be a vector we denote by
  $G_{e}$ the group whose underlying set is $\TT\times \GA$ and
  multiplication law is defined by
  \[
  (t_{1},\alpha_{1})\cdot (t_{2},\alpha_{2}) = (t_{1}\cdot t_{2},
  \chi^{-e}(t_{2})\cdot \alpha_{1} + \alpha_{2}),
  \]
  where $t_{i}\in\TT$ and $\alpha_{i}\in\GA$.  Actually, every
  semidirect product of $\TT\ltimes\GA$ given by a character
  $\TT\rightarrow \rm Aut\,\it\GA\simeq \GM$ is isomorphic to some
  $G_{e}$.
\end{notation}

The following proposition is similar to \cite[Lemma~2.2]{FZ2}.  For
the convenience of the reader we give a short proof.

\begin{proposition}\label{sec:norm-graded}
  Assume that the field $\KK$ is algebraically closed.
  \begin{enumerate}[$(i)$]

  \item If $A$ is $M$-graded and $\partial$ is a homogeneous LFIHD on
    $A$ of degree $e$, then the corresponding $\GA$-action is
    normalized by the $\TT$-action. This means that the actions of the
    torus and the additive group induce a $G_{e}$-action with
    comorphism given by
    \[
    \psi^{*}(t,\alpha) = t\cdot e^{\alpha\partial}(f),
    \]
    where $(t,\alpha)\in G_{e}$ and $f\in A$.

  \item Conversely, if $G_{e}$ acts on $X = \spec A$, then the actions
    of the subgroups $\TT$ and $\GA$ give an $M$-grading on $A$ and a
    homogeneous LFIHD of degree $e$.

  \item Assume further that $\chara(\KK) = p\rm >0$. Let
    $F_{p^{r}}:\GA\rightarrow\GA$, $t\mapsto t^{p^{r}}$ be the
    Frobenius map. Giving a rationally homogeneous LFIHD
    $\partial$ on $A$ of degree $e/p^{r}$ is equivalent to having a
    $\GA$-action on $X$ equal to $\phi\circ
    (F_{p^{r}},\operatorname{id}_{X})$, where $\phi$ is a $\GA$-action
    normalized by $\TT$.
  \end{enumerate}
\end{proposition}

\begin{proof}
  $(i)$ Given $(t,\alpha)\in G_{e}$ and $f\in A$, by homogeneity of
  $\partial$ we have
  \begin{align} \label{eq:1}
    t\cdot\partial^{(i)}(f) =
    \chi^{ie}(t)\,\partial^{(i)}(t\cdot f),\,\,\forall i\in\ZZ_{\geq 0}. 
  \end{align}
  This gives
  \[
  t\cdot e^{\alpha\partial}(f) = \sum_{i =
    0}^{\infty}\chi^{ie}(t)\alpha^{i} \,\partial^{(i)}(t\cdot f) =
  e^{\chi^{e}(t)\alpha\partial}(t\cdot f).
  \]
  Hence for all $(t_{1},\alpha_{1}), (t_{2},\alpha_{2})\in G_{e}$ we
  obtain
  \[
  \psi^*((t_{1},\alpha_{1})\cdot (t_{2},\alpha_{2}))(f) =
  e^{\chi^{e}(t_{1})\alpha_{1}
    \partial}\circ
  e^{\chi^{e}(t_{1}t_{2})\alpha_{2}\partial}(t_{1}t_{2}\cdot f) =
  \psi^*(t_{1},\alpha_{1})(\psi^*(t_{2},\alpha_{2})(f)).
  \]
  We conclude that $\psi^*$ defines a $G_{e}$-action on the variety
  $X = \spec A$.

$(ii)$ The action of the subgroup $\GA\subseteq G_{e}$ yields
an LFIHD $\partial$ on the algebra $A$.
For  $\alpha\in \GA$ and $f\in A$ we have 
$\psi^*(1,\alpha)(f) = e^{\alpha\partial}(f)$.
So for any $t\in\TT$ we have 
\[
t\cdot e^{\alpha\partial}(f) = 
\psi^*((1,\chi^{e}(t)\alpha)\cdot (t,0)) (f) = 
e^{\chi^{e}(t)\alpha\partial}(t\cdot f).
\]
Identifying the coefficients we obtain \eqref{eq:1}.  Thus the LFIHD
$\partial$ is homogeneous for the $M$-grading given by the action of
the subgroup $\TT\subseteq G_{e}$.

Assertion $(iii)$ follows immediately from $ (i)$ and $(ii)$.
\end{proof}

For an arbitrary field $\KK$ we consider the following natural definition.

\begin{definition}
  Assume that the torus $\TT$ acts on $X = \spec A$.  A $\GA$-action
  on $X$ is \emph{normalized} (resp.  \emph{normalized up to a
    Frobenius map}) by the $\TT$-action if the corresponding LFIHD
  $\partial$ is homogeneous (resp. rationally homogeneous).
\end{definition}

To classify normalized $\GA$-action it is convenient to
separate them into two types (see \cite[3.11]{FZ2} and
\cite[Lemma~1.11]{Li} for special cases).

\begin{definition}
  A homogeneous LFIHD $\partial$ is of \emph{vertical type} (or of
  fiber type) if $\bar{\partial}^{(i)}(K_{0}) = \{0\}$ for any
  $i\in\ZZ_{>0}$.  Otherwise $\partial$ is of \emph{horizontal
    type}. We use similar terminology for normalized $\GA$-actions. An
  affine $\TT$-variety endowed with a non-trivial vertical
  (resp. horizontal) $\GA$-action is called \emph{vertical}
  (resp. \emph{horizontal}).
\end{definition}

A homogeneous LFIHD of horizontal type is automatically non-trivial.
In the vertical case, one can extend a homogeneous LFIHD on $A$ to an
LFIHD on the semigroup algebra $K_{0}[\sigma^{\vee}_{M}]$.

\begin{lemme} \label{sec:lem2.12}
  Let $\partial$ be a homogeneous LFIHD of vertical type on the
  $M$-graded algebra $A$.  Then $\partial$ extends to a unique
  homogeneous locally finite iterative higher $K_{0}$-derivation on
  the semigroup algebra $K_{0}[\sigma^{\vee}_{M}]$.
\end{lemme}

\begin{proof}
  By Lemma~\ref{sec:mult-syst}, the LFIHD $\partial$ extends to an
  iterative higher derivation $\partial'$ on $K_{0}[M]$. Since
  $\partial$ is of vertical type, Definition~\ref{sec:sit2.1}~$(ii)$
  implies that each $\partial'^{(i)}$ is $K_{0}$-linear. Consequently,
  if $S\subseteq M$ is the subsemigroup of weights of the $M$-graded
  algebra $A$, then $B:= K_{0}[S] = A\otimes_{\KK}K_{0}$ is
  $\partial'$-invariant.

  Let us show that $\partial'|_{B}$ is an LFIHD on $B$.  Let
  $f\chi^{m}\in B$ be a homogeneous element with $f\in K_{0}^{*}$.
  Write $f\chi^{m} = f'h\chi^{m}$ for some $f'\in K_{0}$ and for some
  $h\in A_{m}$. There exists $r\in\ZZ_{>0}$ such that for any $i\geq
  r$,
  \[
  \partial'^{(i)}(f\chi^{m}) = f'\partial^{(i)}(h\chi^{m}) = 0.
  \]
  Since every element of $B$ is a sum of homogeneous elements we
  conclude that $\partial'|_{B}$ is a locally finite iterative higher
  $K_{0}$-derivation on $B$. Thus, $\partial'|_{B}$ extends to an
  LFIHD on the integral closure $\bar{B} = K_{0}[\sigma^{\vee}_{M}]$.
\end{proof}

In the next lemma, we prove an elementary result concerning the LFIHDs
of the polynomial algebra in one variable. It will be useful in order
to study horizontal $\GA$-actions in Section~\ref{sec5}. We let
$\ord_0$ be the natural valuation
\[
\ord_0:\KK[t]\setminus\{0\}\rightarrow\ZZ_{\geq 0},\quad
\sum_{i}a_{i}t^{i}\mapsto \min\{i\,|\, a_{i}\neq 0\}\,.
\]

\begin{lemme}\label{sec:line}
  Assume that $\chara(\KK)=p>0$. Let $\partial$ be an LFIHD on the
  polynomial algebra $\KK[t]$ in one variable such that
  \[
  e^{x\partial}(t) = t + \sum_{i = 1}^{r}\lambda_{i}x^{p^{s_{i}}},
  \]
  where $\lambda_{i}\in\KK^{*}$ and $0\leq s_{1}<\ldots <s_{r}$ are
  integers. We also fix a
  non-negative integer $i\in\ZZ_{\geq 0}$.

  If $\ell\in\ZZ_{\geq 0}$ verifies $\ell\geq ip^{s_{1}}$, then
  \[
  \partial^{(ip^{s_{1}})}(t^{\ell}) =
  \lambda_{1}^{i}\binom{\ell}{i}t^{\ell - i}
  \]
  and therefore $\ord_0\partial^{(ip^{s_{1}})}(t^{\ell}) = \ell - i$
  whenever $\binom{\ell}{i}\neq 0$.
\end{lemme}
\begin{proof}
  First of all, we have
  \[
  e^{x\partial}(t^{\ell}) = e^{x\partial}(t)^{\ell} = \left(t +
    \sum_{i = 1}^{r}\lambda_{i}x^{p^{s_{i}}}\right)^{\ell} =
  \sum_{i_{0} + \ldots + i_{r} = \ell,\,i_{0},\ldots,i_{r}\geq 0}
  \binom{\ell}{i_{0}\ldots i_{r}}t^{i_{0}}\prod_{\alpha =
    1}^{r}(\lambda_{\alpha}x^{p^{s_{\alpha}}})^{i_{\alpha}}.
  \]
  Considering the term of degree $ip^{s_{1}}$ in $x$ of the previous
  sum, we get the following conditions:
  \begin{align}\label{eq:2}
  ip^{s_{1}} = i_{1}p^{s_{1}} + \ldots + i_{r}p^{s_{r}}\quad\mbox{ and
  }\quad i_{0} + i_{1} + \ldots + i_{r} = \ell,
  \end{align}
  where $(i_{0},i_{1},\ldots, i_{r})\in\ZZ_{\geq 0}^{r+1}$. Note that
  such a $(r+1)$-tuple $(i_{0},i_{1},\ldots, i_{r})$ exists since
  $\ell\geq ip^{s_{1}}$ and so we can take
  \[
  (i_{0},i_{1},\ldots, i_{r}) = (\ell-i,i,0,\ldots, 0).
  \]
  Let us show that this is the minimal choice for $i_{0}\in\ZZ_{\geq
    0}$.  Indeed, let $(\gamma_{0},\gamma_{1},\ldots,
  \gamma_{r})\in\ZZ_{\geq 0}^{r}$ be an $(r+1)$-uplet
  satisfying~\eqref{eq:2} with $\gamma_0$ minimal. Then we have
  \[
  \ell - i = \ell - \sum_{\alpha = 1}^{r}\gamma_{\alpha}p^{s_{\alpha} -
    s_{1}}\leq \ell - \sum_{\alpha = 1} \gamma_{\alpha} = \gamma_{0}.
  \]
  Hence by minimality, $\gamma_{0} = \ell - i$, so that $i =
  \sum_{\alpha = 1}^{r}\gamma_{\alpha}$.  Thus,
  \[
  \left( \sum_{\gamma_{\alpha}}^{r}\gamma_{\alpha}\right) p^{s_{1}} =
  \sum_{\alpha = 1}^{r}\gamma_{\alpha}p^{s_{\alpha}}.
  \]
  We obtain $(\gamma_{0},\gamma_{1},\ldots, \gamma_{r}) =
  (\ell-i,i,0,\ldots, 0)$.  This implies in particular that
  $\partial^{(ip^{s_{1}})}(t^\ell) = \lambda_{1}^{i}\binom{\ell}{i}t^{\ell -
    i}$ as required.
\end{proof}

\section{$\GA$-actions on affine toric varieties} \label{sec3}

Let $\KK$ be a field.  In this section, we present a combinatorial
description of normalized $\GA$-actions up to a Frobenius map on
affine toric varieties over $\KK$.

For a rational cone $\sigma\subseteq N_{\RR}$ we recall that $\sigma(1)$
denotes its set of extremal rays. As usual we write by the same letter
a ray of $\sigma$ and its primitive vector.  The following is a
classical definition, see for instance \cite{Dem,Li,AL}. 

\begin{definition}
  Let $\sigma\subseteq N_{\RR}$ be a strongly convex rational cone. A
  vector $e\in M$ is called a \emph{Demazure's root} (or for
  simplicity called \emph{root}) if the following hold.
  \begin{enumerate}[$(i)$]
  \item There exists $\rho\in \sigma(1)$ such that $\langle
    e,\rho\rangle =-1$.
  \item For any $\rho'\in \sigma(1)\setminus\{\rho\}$ we have $\langle
    e,\rho'\rangle\geq 0$.
  \end{enumerate}
  The extremal ray $\rho$ satisfying $\langle e,\rho\rangle =-1$ is
  called the \emph{distinguished ray} of the root $e\in M$. We denote
  by $\rt\sigma$ the set of Demazure's roots of the cone
  $\sigma$.  By \cite[Remark~2.5]{Li} every element of $\sigma(1)$ is
  the distinguished ray of a root of $\rt\sigma$.
\end{definition}

Since the subset $\KK[\TT]^{*}$ generates the algebra $\KK[\TT]$,
Proposition~\ref{sec:LFIHD}~$(e)$ implies that $\KK[\TT]$ has no
non-trivial LFIHDs. So without loss of generality, in the sequel, we
may only consider toric varieties $X_{\sigma} =
\spec\KK[\sigma^{\vee}_{M}]$ given by a nonzero strongly convex
rational cone $\sigma\subseteq N_{\RR}$.

\begin{exemple} \label{sec:ex3.2}
  Let $e\in\rt \sigma$ be a root. Consider the homogeneous derivation
  $\partial_e^{(1)}$ on the semigroup algebra $\KK[\sigma^{\vee}_{M}]$ given
  by
  \[
  \partial_{e}^{(1)}(\chi^{m}) = \langle m,\rho\rangle
  \chi^{m+e}\quad\mbox{for all}\quad m\in\sigma^\vee_M\,,
  \]
  where $\rho$ is the distinguished ray of $e$. Then
  $\partial_{e}^{(1)}$ is locally nilpotent and yields a $\GA$-action
  on $X_{\sigma}$ in the following natural way: the homogeneous LFIHD
  $\partial_e$ is given by the formula\footnote{We set the
    convention that $\binom{r_{1}}{r_{2}} = 0$, for all $r_{1}$,
    $r_{2}\in\ZZ_{\geq 0}$ with $r_{1}<r_{2}$.}
  \[
  \partial_{e}^{(i)}(\chi^{m}) = \binom{\langle
    m,\rho\rangle}{i}\cdot\chi^{m+ie}\quad\mbox{for all}\quad i \in
  \ZZ_{\geq 0} \quad\mbox{and}\quad m\in\sigma^\vee_M\,.
  \]
  The kernel of $\partial_e$ is $\KK[\rho^{\star}_M]$, where
  $\rho^{\star}\subseteq \sigma^{\vee}$ is the dual face of $\rho$.

  Assume now that $\chara(\KK) = p>0$.  Starting from $\partial_{e}$
  and an integer $r\in\ZZ_{\geq 0}$ we can also define a rationally homogeneous
  LFIHD $\partial_{e,r}$ of degree $e/p^{r}\in M_\QQ$. Its exponential
  map is
  \[
  e^{x\partial_{e,r}} = \sum_{i =
    0}^{\infty}\partial_{e}^{(i)}\,x^{ip^{r}}\,.
  \]
  We check easily that $\ker \partial_{e,r} = \KK[\rho^{\star}_M]$.
  In addition, for any $m\in\sigma^{\vee}_{M}$ we have
  \[
  \deg_{x} e^{x\partial_{e,r}}(\chi^{m}) = p^{r}\langle
  m,\rho\rangle.
  \]
\end{exemple}

We start by describing the kernel and the possible degree vectors
of a homogeneous LFIHD on $\KK[\sigma^{\vee}_{M}]$, where $\sigma$ is
a nonzero strongly convex rational cone.

\begin{lemme} \label{sec:lem3.3}
  Consider a non-trivial homogeneous LFIHD $\partial$ on
  $\KK[\sigma^{\vee}_{M}]$. Then the following statements hold.
  \begin{enumerate}[$(i)$]
  \item There exists $\rho\in\sigma(1)$ such that $\rm \ker \partial =
    \KK[\rho^{\star}\cap M]$.
  \item The degree $e\in M$ of the sequence $\partial$ is a Demazure's
    root of $\sigma$ and $\rho$ is the distinguished ray of $e$.
  \end{enumerate}
\end{lemme}

\begin{proof}
  $(i)$ By Proposition~\ref{sec:LFIHD}~$(a)$ we have $\ker\partial =
  \KK[W\cap\sigma^{\vee}_{M}]$ for some linear subspace $W\subseteq
  M_{\RR}$.  Assume that $W\cap\sigma^{\vee}$ is not a face of
  $\sigma^{\vee}$. Then $W$ divides $\sigma^{\vee}$ into two parts. We
  can find $m\in\sigma^{\vee}_{M}$ such that for any $r\in\ZZ_{\geq
    0}$, $m+re\not\in W$. Since $\chi^{m}\not\in \ker\partial$, there
  is some $r_{0}\in\ZZ_{>0}$ satisfying
  $\partial^{(r_{0})}(\chi^{m})\neq 0$.  Hence
  $\partial^{(r_{0})}(\chi^{m})$ is homogeneous of degree
  $m+r_{0}e$. By the previous argument
  \[
  \partial^{(r'_{1})}\circ\partial^{(r_{0})}(\chi^{m})\neq 0
  \quad\mbox{for some}\quad r'_{1}\in\ZZ_{>0}\,.
  \]

  By Definition~\ref{sec:sit2.1}~$(iv)$ we have
  $\partial^{(r_{0}+r_{1}')}(\chi^{m})\neq 0$ and so we let $r_{1} =
  r_{0}+r'_{1}$.  Proceeding by induction we can build a strictly
  increasing sequence of positive integers $\{r_{j}\}_{j\in\ZZ_{\geq 0}}$
  verifying $\partial^{(r_{j})}(\chi^{m})\neq 0$ for any
  $j\in\ZZ_{\geq 0}$. This contradicts the fact that $\partial$ is an
  LFIHD. Thus $W\cap \sigma^{\vee}$ is a face of
  $\sigma^{\vee}$. Since $\rm ker\,\partial$ is a subring of
  codimension one, we have $W\cap\sigma^\vee_M=\rho^\star\cap M$ for
  some extremal ray $\rho\in\sigma(1)$.

  $(ii)$ If $e\in\sigma^{\vee}_{M}$, then the same argument as before
  gives a contradiction.  The rest of the proof follows as in
  \cite[Lemma~2.4]{Li}.
\end{proof}

In the following lemma, we state some properties of a homogeneous
LFIHD on $\KK[\sigma^{\vee}_{M}]$.

\begin{lemme} \label{sec:lem3.4} Let $\partial$ be a non-trivial
  homogeneous LFIHD on $\KK[\sigma^{\vee}_{M}]$ of degree $e$ and with
  distinguished ray $\rho$. For every $i\in\ZZ_{\geq 0}$ we let
  $c_i:\sigma^\vee_M\rightarrow \KK$ be such that
  $ \partial^{(i)}(\chi^{m}) = c_{i}(m)\chi^{m+ie}$.  Then the
  sequence $\{c_{i}\}_{i\in\ZZ_{\geq 0}}$ of functions on
  $\sigma^{\vee}_{M}$ satisfies the following conditions.
  \begin{enumerate}[$(i)$]
  \item The map $c_{0}$ is the constant map $m\mapsto 1$.
  \item For all $m,m'\in\sigma^{\vee}_{M}$ we have
    \begin{align}\label{eq:3}
      c_{i}(m+m') = \sum_{j = 0}^{i}c_{i-j}(m)\cdot c_{j}(m'). 
    \end{align}
  \item For every $m\in\sigma^{\vee}_M$ there exists
    $r\in\ZZ_{\geq 0}$ such that $c_{i}(m)=0$ for all $i\geq r$.
  \item For every $i,j\in\ZZ_{\geq 0}$ we have
    \[
    \binom{i+j}{i}\,c_{i+j}(m) = c_{i}(m+je)\cdot
    c_{j}(m)\quad\mbox{for all}\quad m\in\sigma^{\vee}_{M}\,.
    \]
  \item For every $i\in \ZZ_{\geq 0}$ we have $c_{i}(m+m') = c_{i}(m)$ for all
    $m\in\sigma^{\vee}_{M}$ and all $m'\in\rho^{\star}\cap M$.
  \end{enumerate}
\end{lemme}

\begin{proof}
  Assertions $(i)$, $(ii)$, $(iii)$ and $(iv)$ follow from the
  definition of LFIHD. Let us show $(v)$.  Since $\chi^{m'}\in
  \ker\partial$, for any $j\in\ZZ_{>0}$ we have $c_{j}(m') =
  0$. Applying \eqref{eq:3} we obtain $c_{i}(m+m') = c_{i}(m)$.
\end{proof}

The next result provides a classification of normalized
$\GA$-actions on $X_{\sigma}$. See \cite[Theorem~2.7]{Li}
for the case where $\chara(\KK)=0$.

\begin{theorem}\label{sec:th3.5}
  Let $\sigma\subseteq N_{\RR}$ be a nonzero strongly convex rational
  cone.  Every non-trivial $\GA$-action on $X_{\sigma}$ normalized by
  the $\TT$-action is given by a homogeneous LFIHD of the form
  $\lambda\partial_{e}$, where $\partial_e$ is as in
  Example~\ref{sec:ex3.2}, $e\in \rt \sigma$ and $\lambda\in\KK^{*}$.
\end{theorem}

\begin{proof}
  Let $\partial$ be a non-trivial homogeneous LFIHD of degree $e$ on
  $\KK[\sigma^{\vee}_{M}]$.  By Lemma~\ref{sec:lem3.3}, $e$ is a root
  of $\sigma$ and $\ker\partial = \KK[\rho^{\star}\cap M]$, where
  $\rho\in\sigma(1)$ is the distinguished ray of the root $e$.

  Let us first show that there exists a lattice vector
  $m\in\sigma^{\vee}_{M}$ such that $\langle m, \rho\rangle = 1$. Let
  $m'\in\sigma^{\vee}_{M}$ not contained in the face $\rho^{\star}$ so
  that $\langle m',\rho\rangle > 1$. By \cite[Lemma 2.4]{Li}, we have
  that $m := m' + (\langle m',\rho\rangle-1)\cdot
  e\in\sigma^{\vee}_{M}$ satisfies $\langle m,\rho\rangle = 1$.

  We let $c_i:\sigma^\vee_M\rightarrow \KK$ be the maps defined in
  Lemma~\ref{sec:lem3.4}.  Let $B = \KK[x]$ be the polynomial algebra
  of one variable. Using the basis $(1,x,x^{2},\ldots)$ we define a
  sequence of linear operators
  $\bar{\partial}=\{\bar{\partial}^{(i)}\}_{i\in\ZZ_{\geq 0}}$ on the
  $\KK$-linear space $B$ as follows: fixing a vector
  $m\in\sigma^{\vee}_{M}$ verifying $\langle m,\rho\rangle = 1$ we
  define 
  \[
  \bar{\partial}^{(i)}(x^{r}) = c_{i}(rm)x^{r-i}\quad\mbox{for
    all}\quad i,r\in \ZZ_{\geq 0}\,.
  \]
  We claim that $\bar{\partial}$ is well defined. Indeed, let
  $i,r\in\ZZ_{\geq 0}$ be such that $i>r$, then
  \[
  \partial^{(i)}(\chi^{rm}) = c_{i}(rm)\chi^{rm+ie}\in
  \KK[\sigma^{\vee}_{M}]\quad\mbox{and}\quad \langle rm+ie, \rho
  \rangle = r-i < 0\quad\mbox{so that}\quad c_{i}(rm) = 0\,.
  \]
  Hence, $\bar{\partial}^{(i)}(x^{r}) = c_{i}(rm)x^{r-i}=0$ for all
  $i>r$.

  By Lemma~\ref{sec:lem3.4}, the sequence of operators
  $\bar{\partial}$ is an LFIHD on $B$.  For instance, let us show that
  $\bar{\partial}$ satisfies Definition~\ref{sec:sit2.1}~$(iv)$.
  Letting $i,j\in\ZZ_{\geq 0}$ we have
  \[
  \bar{\partial}^{(i)}\circ \bar{\partial}^{(j)}(x^{r}) =
  \bar{\partial}^{(i)}(c_{j}(rm)x^{r-j}) = c_{i}((r-j)m)\cdot
  c_{j}(rm)x^{r-i-j}.
  \]
  Since $e\in \rt\sigma$ is a root having $\rho$ as distinguished ray,
  it follows that
  \[
  v:= rm + je - (r-j)m = j(m+e)\in\rho^{\star}\cap M.
  \]
  By Lemma~\ref{sec:lem3.4}~$(v)$, we have
  \[
  c_{i}((r-j)m) = c_{i}((r-j)m + v) = c_{i}(rm + je).
  \]
  Therefore by Lemma~\ref{sec:lem3.4}~$(iv)$, we conclude that
  \[
  \bar{\partial}^{(i)}\circ\bar{\partial}^{(j)}(x^{r}) =
  \binom{i+j}{i}c_{i+j}(rm)x^{r-i-j} =
  \binom{i+j}{i}\bar{\partial}^{(i+j)}(x^{r}),
  \]
  as required. Conditions $(i),(ii),(iii)$ of
  Definition~\ref{sec:sit2.1} follow from similar straightforward
  computations.

  Since $\bar{\partial}$ is homogeneous for the natural graduation of
  $B$, by Proposition~\ref{sec:LFIHD}~$(d)$ there exists
  $\lambda\in\KK$ such that every $c_{i}$ verifies
  \[
  c_{i}(rm) = \binom{r}{i} \,\lambda^{i}
  \]
  for any $r\in\ZZ_{\geq 0}$. We use the convention $\lambda^{0} = 1$ whenever
  $\lambda = 0$. Let $w\in\sigma^{\vee}_{M}$ be a lattice vector. The
  elements
  \[
  w + \langle w,\rho\rangle e, \, \langle w,\rho\rangle e + \langle
  w,\rho\rangle m
  \]
  belong to $\rho^{\star}\cap M$. By Lemma $3.4\,\rm (v)$ this implies
  \begin{align}\label{eq:4}
    c_{i}(w) = c_{i}\left( w + \langle w,\rho\rangle e + \langle w,\rho\rangle m\right )=
    c_{i}\left (\langle w,\rho\rangle m\right ) = \binom{\langle w,\rho\rangle}{i}\,\lambda^{i}.
  \end{align}
  Since $\partial$ is non-trivial, we have $\lambda\in\KK^{*}$. By
  virtue of \eqref{eq:4} the sequence $\partial$ is given by the LFIHD
  $\lambda\partial_{e}$ (see Example $3.2$).
\end{proof}

\begin{exemple}
  Let $M=\ZZ^2$ and let $\sigma$ be the strongly convex rational cone
  generated in the vector space $N_\RR=\RR^2$ by the vectors and
  $\rho=(0,1)$ and $\rho'=(2,-1)$. The dual cone $\sigma^\vee$ is the
  cone in $M_\RR$ generated by the vectors $(1,0)$ and $(1,2)$. Let
  $A=\KK[\sigma^\vee_M]$ and let $X=\spec A$ be the corresponding
  toric variety. The algebra $A$ is generated by the elements
  \[u_1=\chi^{(1,0)},\quad u_2=\chi^{(1,1)}\quad\mbox{and}\quad
  u_3=\chi^{(1,2)}\,.\] The generators satisfy the relation
  $u_1u_3=u_2^2$ and so $A=\KK[x,y,z]/(xz-y^2)$. The lattice vector
  $e=(0,-1)\in M$ is a root of $\sigma$ since $\langle
  e,\rho\rangle=-1$ and $\langle e,\rho'\rangle=1$.

  
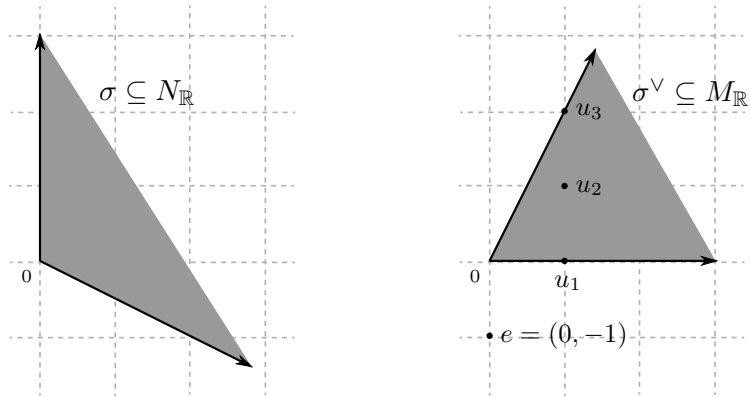
\begin{figure}[!ht]
    \input{toric.tex}
  \end{figure}

  The corresponding LFIHD $\partial_e$ of Example~\ref{sec:ex3.2} is
  given by
  \begin{align*}
    &\partial_e^{(0)}(x)=x,\quad\partial_e^{(i)}(x)=0,\quad \mbox{for
      all } i>0\,; \\
    &\partial_e^{(0)}(y)=y,\quad \partial_e^{(1)}(y)=x,\quad
    \partial_e^{(i)}(y)=0, \quad \mbox{for all }i>1\,; \\
    &\partial_e^{(0)}(z)=z,\quad \partial_e^{(1)}(z)=2y,\quad
    \partial_e^{(2)}(z)=x,\quad\partial_e^{(i)}(z)=0,\quad \mbox{for
      all }i>2\,.
  \end{align*}
  Hence, the corresponding normalized $\GA$-action $\phi$ is defined
  by
  \[\phi:\GA\times X\rightarrow X,\quad\mbox{where}\quad
  (\lambda,(x,y,z))\mapsto (x,y+\lambda x,z+2\lambda y+\lambda^2z)\,.\]
\end{exemple}

As an immediate consequence of Theorem~\ref{sec:th3.5}, we obtain a
description of all normalized $\GA$-actions up to a Frobenius map.

\begin{corollaire}\label{cor3.6}
  Let $\sigma\subseteq N_{\RR}$ be a nonzero strongly convex rational
  cone.  Then for every root $e\in \rt\sigma$ with distinguished ray
  $\rho$, every integer $r\in\ZZ_{\geq 0}$, and every scalar
  $\lambda\in\KK^{*}$, there is a non-trivial rationally homogeneous
  LFIHD $\partial$ on the algebra $\KK[\sigma^{\vee}_{M}]$ whose
  exponential is given by
  \[
  e^{x\partial}(\chi^{m}) = \sum_{i = 0}^{\infty}\binom{\langle
    m,\rho\rangle}{i}\lambda^{i}\,\chi^{m+ie}x^{ip^{r}}\quad\mbox{for
    all}\quad m\in\sigma^{\vee}_{M}\,.
  \]
  Conversely, every rationally homogeneous LFIHD on
  $\KK[\sigma^{\vee}_{M}]$ arises in this way.
\end{corollaire}

In the next corollary, we generalize to the case of positive
characteristic some results in \cite[Section~2]{Li}.  See also
\cite[Corollary~3.5]{Ku} for a more general statement in the
characteristic zero case. The proofs are similar to those in \cite{Li}
so we omit them.

\begin{corollaire}
  Let $\sigma\subseteq N_{\RR}$ be a strongly convex rational, then the
  following hold.
  \begin{enumerate}[$(i)$]
  \item For any normalized up to a Frobenius map $\GA$-actions in
    $\spec \KK[\sigma^\vee_M]$ the algebra of invariants is finitely
    generated.
  \item There is a finite number of rationally homogeneous
    LFIHDs on $\KK[\sigma^{\vee}_{M}]$ with pairwise distinct
    kernels.
  \end{enumerate}
\end{corollaire}

\section{$\GA$-actions of vertical type} \label{sec4}

In this section, we classify normalized $\GA$-actions of vertical type
on an affine $\TT$-variety $X=\spec A$ of complexity one over a field
$\KK$.  See \cite{Li2} for higher complexity when the base field is
algebraically closed of characteristic zero.

To achieve our classification, we place ourselves in the context of
Section~\ref{sec1} by letting $A=A[C,\DD]$, where $C$ is a regular
curve over $\KK$ and $\DD = \sum_{z\in C}\Delta_{z}\cdot z$ is a
proper $\sigma$-polyhedral divisor. Hence,
\[
A[C,\DD] = \bigoplus_{m\in\sigma^{\vee}_{M}}A_m\cdot\chi^{m}\subseteq
K_{0}[M],\quad\mbox{where}\quad A_m=H^{0}\big(C,
\OO_{C}(\DD(m))\big)\mbox{ and } K_{0} = \KK(C)\,.
\]

The following result gives some general properties of homogeneous
LFIHDs on the $M$-graded algebra $A$. Recall that the affine
$\TT$-variety $X=\spec A$ is called elliptic if $A_0=\KK$.

\begin{lemme} \label{4.1}
  Let $\partial$ be a homogeneous $\rm LFIHD$ on $A$ of degree $e$.
  Then the following statements hold.
  \begin{enumerate}[$(i)$]
  \item If $\partial$ is vertical, then $e\not\in\sigma^{\vee}$ and
    $\ker\partial = \bigoplus_{m\in\tau_{M}}A_{m}\chi^{m}$ for some
    codimension one face $\tau$ of the cone $\sigma^{\vee}$. In
    particular, the algebra $\rm ker \,\it\partial$ is finitely
    generated.
  \item If $A$ is non-elliptic, then $\partial$ is vertical if and only
    if $e\not\in\sigma^{\vee}$.
  \end{enumerate}
\end{lemme}

\begin{proof}
  $(i)$ By Lemma~\ref{sec:lem2.12} we may extend $\partial$ to a
  homogeneous LFIHD $\bar{\partial}$ on the semigroup $K_{0}$-algebra
  $K_{0}[\sigma^{\vee}_{M}]$. By Lemma~\ref{sec:lem3.3} we have
  $e\in\rt\sigma$ and so $e\not\in\sigma^{\vee}$. Moreover, we obtain
  $\ker\bar{\partial} = K_{0}[\tau_{M}]$ for some codimension one face
  $\tau$ of $\sigma^{\vee}$.  Thus,
  \[
  \ker\partial = A\cap\ker\bar{\partial} =
  \bigoplus_{m\in\tau_{M}}A_{m}\chi^{m}.
  \]
  As a consequence of \cite[Lemma~4.1]{AH}, the algebra $\ker\partial$

  is finitely generated.
  $(ii)$ Assume that $A$ is non-elliptic and let $\bar{\partial}$ be
  the extension of $\partial$ on the $K_{0}$-algebra $K_{0}[M]$. If
  $e\not\in\sigma^{\vee}$, then for any $i\in\ZZ_{> 0}$ we have
  $\partial^{(i)}(A_{0}) = A_{ie} = \{0\}$. Since $K_{0} = \fract
  A_{0}$, we conclude that $\partial$ is vertical.
\end{proof}

As remarked in \cite[Remark~3.2]{Li}, in the elliptic case, the
$M$-graded algebra admits in general LFIHDs $\partial$ of horizontal
type satisfying $\deg\partial\not\in\sigma^{\vee}$.

In the following, we introduce some combinatorial data on $A=
A[C,\DD]$ in order to describe its vertical normalized $\GA$-actions.

\begin{notation}
  Let $e\in\rt\sigma$ be a root of $\sigma$ with distinguished ray
  $\rho$ and recall that $ \DD(e) = \sum_{z\in C}\min_{v\in
    \Delta_{z}(0)}\langle e, v\rangle\cdot z$. We denote by $\Phi_{e}$
  the $A_{0}$-module $H^{0}(C,\OO_{C}(\DD(e)))$. Furthermore, if
  $\varphi\in\Phi_{e}$ is a nonzero section, then for any vector
  $m\in\sigma^{\vee}$ belonging to $M_{\QQ}$ we have
  \begin{align}
    \label{eq:5}
    \divi\varphi \geq -\DD(e)\geq\DD(m) - \DD(m+e).
  \end{align}
\end{notation}

Starting with the previous combinatorial data, we may construct a
homogeneous LFIHD of vertical type, as follows:

\begin{lemme}\label{sec:lem4.3}
  Let $e\in \rt\sigma$ be a root of $\sigma$ with distinguished ray
  $\rho$ and let $\varphi\in \Phi_{e}$ be a section. Denote
  $\bar{\partial} = \varphi\,\partial_{e}$, where $\partial_{e}$ is
  the $\rm LFIHD$ on the $K_{0}$-algebra $K_{0}[\sigma^{\vee}_{M}]$
  corresponding to the root $e$ as in Example $3.2$. Then for any
  $i\in\ZZ_{\geq 0}$ we have $\bar{\partial}^{(i)}(A)\subseteq A$. Consequently,
  the sequence
  \[
  \partial_{e,\varphi} :=
  \left\{\bar{\partial}^{(i)}|_{A}:A\rightarrow A\right\}_{i\in\ZZ_{\geq 0}}
  \]
  defines a homogeneous $\rm LFIHD$ of vertical type on $A$.
\end{lemme}

\begin{proof}
  Fix $i\in\ZZ_{>0}$ and let $f\in A_{m}$ be nonzero such that $\divi
  f + \lfloor \DD(m)\rfloor \rm \geq 0$. If
  $\partial^{(i)}(f\chi^{m})\neq 0$ and $\varphi\neq 0$, then by
  \eqref{eq:5} we have
  \begin{align*}
    \divi\left(\partial^{(i)}(f\chi^{m})/\chi^{m+ie}\right) + \lfloor
    \DD(m+ie)\rfloor &= i\divi \varphi + \divi f + \lfloor
    \DD(m+ie)\rfloor \\
    &\geq i(\DD(m/i) - \DD(m/i + e)) - \lfloor \DD(m)\rfloor + \lfloor
    \DD(m+ie)\rfloor \\
    &\geq \{\DD(m)\} - \{\DD(m+ie)\}.  \\
  \end{align*}
  
  Since the coefficients of the $\QQ$-divisor $\{\DD(m)\} -
  \{\DD(m+ie)\}$ belong to $]-1,1[$ we have
  \[
  \divi\left(\partial^{(i)}(f\chi^{m})/\chi^{m+ie}\right) + \lfloor \DD(m +
  ie)\rfloor \geq 0,
  \]
  proving that $A$ is $\partial$-invariant.  The rest of the proof is
  straightforward and left to the reader.
\end{proof}

Our next theorem gives a classification of normalized vertical
$\GA$-actions on an affine $\TT$-variety $X=\spec A[C,\DD]$ of
complexity one.

\begin{theorem} \label{sec:th4.4}
  Let $A = A[C,\DD]$. If $e\in \rt\sigma$ is a root of $\sigma$ with
  distinguished ray $\rho$ and $\varphi\in \Phi_{e}$ is a section,
  then $\partial_{e,\varphi}$ is a homogeneous vertical LFIHD on
  $A$. Conversely, every homogeneous vertical LFIHD on $A$ is of the
  form $\partial_{e, \varphi}$, where $e\in\rt\sigma$ and
  $\varphi\in\Phi_{e}$.
\end{theorem}

\begin{proof}
  The direct implication corresponds to Lemma~\ref{sec:lem4.3}. To
  prove the converse statement, let $\partial$ be a non-trivial
  homogeneous vertical LFIHD on $A$. By Lemma~\ref{sec:lem2.12},
  $\partial$ extends to a locally finite iterative higher
  $K_{0}$-derivation $\bar{\partial}$ on the semigroup algebra
  $K_{0}[\sigma^{\vee}_{M}]$. By Theorem~\ref{sec:th3.5},
  $\bar{\partial}$ is given by $\varphi\partial_{e}$ as in
  Example~\ref{sec:ex3.2}, for some $\varphi\in K_{0}^{*}$ and some
  root $e\in\rt \sigma$.

  To conclude the proof, let us show that $\varphi\in\Phi_{e}$. Let
  $\rho$ be the distinguished ray of $e$. For every point $z\in C$ we
  let $v_{z}$ be a vertex of $\Delta_{z}$ where the minimum
  $\min_{v\in \Delta_{z}(0)}\langle e, v\rangle$ is achieved so
  that
  \[
  \DD(e) = \sum_{z\in C}\langle e, v_{z}\rangle\cdot z\,.
  \]
  For every $z\in C$ we let $ \omega_{z} = \{m\in\sigma^{\vee}\,|\,
  h_{\Delta_z}(m) = \langle m,v_{z}\rangle\}.$ The set
  $\omega_{z}\subseteq M_\RR$ is a full dimensional cone in $M_\RR$ (see
  \cite[Section~1]{AH}).

  Let also $m_{z}\in \sigma^{\vee}_{M} \setminus \rho^{\star}_{M}$ be
  a lattice vector such that $m_{z}$ and $m_{z} + e$ belong to
  $\omega_{z}$, $\deg\DD(m_{z})\geq g$ and $\langle
  m_{z},\rho\rangle\not\in p\ZZ$, where $p$ is characteristic of the
  field $\KK$ and $g$ the genus of the curve $C$. It is always
  possible to choose such $m_{z}$ since $\omega_{z}$ is full
  dimensional, the polyhedral divisor $\DD$ is proper, and the lattice
  vector $\rho$ is primitive. According to the Riemann-Roch Theorem we
  have $A_{m_{z}}\neq \{0\}$.

  Furthermore, the inclusion $
  \partial^{(1)}(A_{m_{z}}\chi^{m_{z}})\subseteq
  A_{m_{z}+e}\chi^{m_{z}+e} $ implies $\varphi A_{m_{z}}\subseteq
  A_{m_{z}+e}$. Consequently, for any $z\in C$ we have
  \[
  \divi\varphi \geq \DD(m_{z}) -\DD(m_{z} +e)\,.
  \]
  The coefficient of the divisor $\DD(m_{z}) -\DD(m_{z} +e)$ at the
  point $z\in C$ is $-\langle v_{z}, e\rangle$. Thus, $ \divi
  \varphi\geq -\DD(e)$ and we have $\varphi\in\Phi_{e}$, as required.
\end{proof}

In analogy with the toric case, the family of vertical normalized
$\GA$-actions on $X =\spec A$ having pairwise distinct kernels is a
finite set. The next result provides a combinatorial criterion for $A$
to admit a homogeneous non-trivial LFIHD of vertical type.

\begin{corollaire}\label{cor:4.5}
  Let $A = A[C,\DD]$ and let $\rho\subseteq\sigma$ be an extremal
  ray. Then, the $M$-graded algebra $A$ admits a non-trivial vertical
  homogeneous LFIHD such that the distinguished ray of
  $e=\deg\partial\in\rt\sigma$ is $\rho$ if and only if one of the
  following conditions holds.
  \begin{enumerate}[$(i)$]
  \item $C$ is affine, or
  \item $C$ is projective and $\rho\cap \deg\DD = \emptyset$.
  \end{enumerate}
\end{corollaire}

\begin{proof}
  If $C$ is an affine curve, then every divisor on $C$ has a global
  nonzero section and so for any $e\in\rt\sigma$ we have $\dim\Phi_{
    e}>0$. In this case, the corollary follows from
  Theorem~\ref{sec:th4.4}.

  Assume that $C$ is projective and fix a root $e\in\rt\sigma$ with
  distinguished ray $\rho$.  Let us notice that for any
  $m\in\rho^{\star}_{M}$ we have $e + m\in\rt\sigma$. Furthermore
  \[
  \DD(e + m)\geq \DD(m) + \DD(e)\quad\mbox{and so}\quad \deg\DD(m +
  e)\geq \deg\DD(m) + \deg\DD(e)\,.
  \]
  Hence, if $\rho\cap \deg\DD = \emptyset$, then we have $\dim\Phi_{e +
    m}>0$ for some $m\in\rho^{\star}_{M}$, by the Riemann-Roch Theorem
  and by the properness of $\DD$.

  Conversely, assume that $\rho\cap \deg\DD\neq \emptyset$.  Since we
  have $\langle e, \rho\rangle = - 1$, there exists a vertex $v$ of
  $\deg\DD$ such that $\langle e, v\rangle < 0$ and therefore
  $\deg\DD(e) < 0$. Under these latter conditions we have
  $\dim\Phi_{e}\rm = 0$. Again, we conclude by
  Theorem~\ref{sec:th4.4} in the case where $C$ is projective.
\end{proof}

\begin{exemple}
  Let the notation be as in Example~\ref{ex:1.8}. Let $\rho$ be the
  ray of $\sigma$ spanned by $(1,0)$ and let $\rho'$ be the ray of
  $\sigma$ spanned by $(0,1)$.  We have $\deg\DD\cap \rho\neq
  \emptyset$ and $\deg\DD\cap \rho'= \emptyset$. Hence,
  Corollary~\ref{cor:4.5} shows that only $\rho'$ can be the
  distinguished ray of the degree $e$ of an LFIHD $\partial$ of
  vertical type.
\end{exemple}

\section{$\GA$-actions of horizontal type} \label{sec5}

The purpose of this section is to classify all horizontal
$\GA$-actions on affine $\TT$-varieties of complexity one over a
perfect field in terms of polyhedral divisors. The reader may consult
\cite[Section~3.2]{Li} for the case where $\KK$ is algebraically
closed and of characteristic zero.  Let as before $A=A[C,\DD]$, where
$C$ is a regular curve over $\KK$ and $\DD = \sum_{z\in
  C}\Delta_{z}\cdot z$ is a proper $\sigma$-polyhedral divisor. Hence,
\[
A[C,\DD] = \bigoplus_{m\in\sigma^{\vee}_{M}}A_m\cdot\chi^{m}\subseteq
K_{0}[M],\quad\mbox{where}\quad A_m=H^{0}\big(C,
\OO_{C}(\DD(m))\big)\mbox{ and } K_{0} = \KK(C)\,.
\]

In this section, several results will require the assumption that
$\KK$ is perfect so the classification will only hold in this
case. Nevertheless, the statements that we can prove without asking
for $\KK$ to be perfect are stated in general.

According to the Rosenlicht Theorem \cite{Ro}, in the case where $\KK$
is algebraically closed, the following lemma implies in particular
that an affine horizontal $\TT$-variety of complexity one has an open
orbit for its corresponding $\TT\ltimes\GA$-action.

\begin{lemme}\label{lm:5.1}
  Let $X=\spec A$, where $A=A[C,\DD]$ and let $\partial$ be a
  homogeneous LFIHD on $A$. Then $\partial$ is horizontal if and only
  if $\KK(X)^{\GA} \cap\KK(X)^{\TT}=\KK$. 
\end{lemme}

\begin{proof}
  Let $L = \KK(X)^{\GA} \cap\KK(X)^{\TT}$.  Assume that $\partial$ is
  horizontal and that $\KK(X)^{\TT}/L$ is an algebraic field
  extension.  Consider $F\in \KK(X)^{\TT}$ a nonzero invariant
  rational function. Remarking that $\KK(X)^{\GA}$ is the field of
  fractions of the ring $\ker\partial$, we can find $a\in
  \ker \partial$ such that $aF$ is integral over $\ker\partial$.
  Since $A$ is normal, $aF\in A$, and by Proposition $2.4(b)$ we have
  $aF\in\ker\partial$.  Hence $F\in \KK(X)^{\GA}$, contradicting the
  fact that $\partial$ is of horizontal type. Since $\KK(X)^\TT/\KK$ is of
  transcendence degree one, we have that $L/\KK$ is algebraic. By our
  convention $\KK$ is algebraically closed in $\KK(X)$ which yields
  $L=\KK$. The converse implication follows directly from the
  definition of horizontal and vertical LFIHDs.
\end{proof}

Our next lemma shows that the existence of a homogeneous LFIHD on the
algebra $A = A[C,\DD]$ imposes some restrictions on the curve $C$. We
refer the reader to \cite[3.5]{FZ2}, \cite[3.16]{Li} for the case
where the base field is algebraically closed of characteristic zero.

\begin{lemme} \label{lem:5.2}
  Assume that $A=A[C,\DD]$ admits a homogeneous LFIHD $\partial$ of
  horizontal type. Consider $\omega$ (resp. $L$ ) the cone
  (resp. sublattice) generated by the weights of $\ker\partial$ and
  let $\omega_{L} = \omega\cap L$.  Then the following statements
  hold.
  \begin{enumerate}[$(i)$]
  \item The kernel of $\partial$ is a semigroup algebra, i.e.,
    \[
    \ker \partial =
    \bigoplus_{m\in\omega_{L}}\KK\cdot\varphi_{m}\chi^{m},\quad\mbox{where}\quad
    \varphi_{m}\in \KK(C)^{*}\,.
    \]

  \item We have $C\simeq \PP^{1}_{\KK}$, in the case where $A$
    is elliptic.

  \item If $\KK$ is perfect, then $C\simeq \AF^{1}_{\KK}$ in the
    case where $A$ is non-elliptic.
  \end{enumerate}
\end{lemme}

\begin{proof}
  $(i)$ Let $a,a'\in\ker\partial\setminus\{0\}$ be homogeneous
  elements of the same degree. By Lemma~\ref{lm:5.1}, we have
  $a/a'\in \KK(X)^{\GA}\cap\KK(X)^{\TT} = \KK$. Thus $\ker\partial$ is
  a semigroup algebra. By Proposition~\ref{sec:LFIHD}~$(b)$ we have
  that $\ker\partial$ is integrally closed, hence normal. This yields
  $(i)$.

  $(ii)$ Let $K = \fract A$ and consider $E = K^{\GA}$.  By
  \cite[Lemma~2.2]{CM} there exists a variable $x$ over the field $E$
  such that $E(x) = K$. By $\rm (i)$, the extension $E/\KK$ is purely
  transcendental and so is $K/\KK$. Since $\KK(C)\subseteq K$, the
  regular projective curve $C$ is unirational. According to the
  Lur\"oth Theorem, it follows that $C\simeq \PP^{1}_{\KK}$.

  $(iii)$ Assume that $A$ is non-elliptic.  Let $\bar{\KK}$ be an
  algebraic closure of $\KK$, so that the field extension
  $\bar{\KK}/\KK$ is separable.  Let $B$ be the algebra
  $A\otimes_{\KK}\bar{\KK}$. Then $B$ is a normal finitely generated
  $M$-graded domain (see Lemma~\ref{lem1.8}).  Note th at the graded
  piece $B_{0}$ is $A_{0}\otimes_{\KK}\bar{\KK}$.  Consequently,
  $\partial$ extends to a homogeneous LFIHD of horizontal type on the
  $\bar{\KK}$-algebra $B$. Now, we can apply the geometrical argument
  in \cite[Lemma~3.16]{Li} to conclude that we have $B_0 \simeq
  \bar{\KK}[t]$, for some variable $t$ over $\bar{\KK}$. By
  separability of $\bar{\KK}/\KK$, this yields $A_{0}\simeq \KK[t]$
  (see e.g. \cite{Ru,As}).
\end{proof}

The preceding lemma implies that the kernel of a horizontal
homogeneous LFIHD on $A$ is finitely generated. This result can be
obtained independently from \cite[Theorem~1.3]{Ku} in the
characteristic zero case. Note also that the kernel of a non-trivial
LFIHD on a normal unirational surface $V$ over a perfect field $\KK$
such that $\KK[V]^*=\KK^*$ is a polynomial algebra (see
\cite[Theorem~2]{Nak78}).

\begin{rappel} \label{sit:5.3}
  In view of the above results, in the following we let $C =
  \AF^{1}_{\KK}$ or $C = \PP^{1}_{\KK}$. Assume that $A$ has a
  homogeneous LFIHD $\partial$ of horizontal type and let
  \[
  \ker \partial = \bigoplus_{m\in\omega_{L}}\KK\cdot\varphi_{m}\chi^{m}
  \]
  be the kernel of $\partial$. We also assume that $\KK(C) = \KK(t)$
  for some local parameter $t$ and, when $C$ is affine, we let $\KK[C]
  = \KK[t]$ be its coordinate ring.
\end{rappel}

\begin{lemme} \label{lm:5.4}
  Keeping the notation as above, the following statements hold.
  \begin{enumerate}[$(i)$]
    
  \item If $C = \AF^1_{\KK}$, then for any $m\in\omega_{L}$ we have
    $\divi \varphi_{m} + \DD(m) = 0$.

  \item Assume that $C = \PP^1_{\KK}$. Then there exists a point
    $z_{\infty}\in C$ such that for any $m\in\omega_{L}$ the effective
    $\QQ$-divisor $\divi \varphi_{m} + \DD(m)$ has at most
    $z_{\infty}$ in its support.
    
  \item The cone $\omega$ is a maximal cone of the quasifan
    $\Lambda\left(\DD\right)$ (see Definition~\ref{sec:def-convex}) in
    the non-elliptic case, and of
    $\Lambda(\DD|_{\PP^{1}_{\KK}\setminus\{z_{\infty}\}})$ for the elliptic
    case.
    
  \item The rank of the lattice $L$ is equal to $n = \rank M$. The
    lattice $M$ is spanned by $e:=\deg\partial$ and $L$. Furthermore,
    if $d$ is the smallest positive integer such that $de\in L$, then
    we can write every vector $m\in M$ in an unique way as $m = l+re$
    for some $l\in L$ and some $r\in\ZZ$ such that $0\leq r< d$.
    
  \item If $\KK$ is perfect, then the point $z_{\infty}$ in $(ii)$ is
    rational, i.e., the residue field of $z_{\infty}$ is $\KK$.
  \end{enumerate}
\end{lemme}

\begin{proof}
  $(i)$ Given a lattice vector $m\in\sigma^{\vee}_{M}$ we let
  \[
  A_{m} = f_{m}\cdot \KK[t]\,,
  \]
  where $f_{m}\in\KK(t)$. Assume that $m\in\omega_{L}$. Then we have
  $\varphi_{m} = Ff_{m}$, for some nonzero $F\in\KK[t]$.  By
  proposition~\ref{sec:LFIHD}$(a)$ the polynomial $F$ is constant. Hence,
  \[
  \divi\varphi_{m} + \lfloor \DD(m)\rfloor = 0.
  \]
  Consequently, for any $r\in\ZZ_{\geq 0}$ we obtain
  \[
  r\lfloor \DD(m)\rfloor = -r\divi \varphi_{m} = - \divi\varphi_{rm} =
  \lfloor \DD(rm)\rfloor.
  \]
  This shows that $\DD(m)$ is integral when $m\in\omega_{L}$.

  $(ii)$ Assume that there exists $m\in\omega_{L}$ such that
  \[
  \divi \varphi_{m} + \DD(m) \geq [z_{\infty}] + [z_{\rm 0}]\,,
  \]
  where $z_{0}$, $z_{\infty}$ are distinct points of $C$. Denote by
  $\infty$ the point at the infinity in $C = \PP^{1}_{\KK}$ for the
  local parameter $t$. Let $p_{0}(t),p_{\infty}(t)\in\KK(t)$ be two
  rational functions verifying the following: if the point $z_{0}$
  (resp. $z_{\infty}$) belongs to $\AF^{1}_{\KK} = \spec \KK[t]$, then
  $p_{0}(t)$ (resp. $p_{\infty}(t)$) is the monic polynomial generator
  of the ideal of $z_{0}$ (resp. $z_{\infty}$) in $\KK[t]$. Otherwise,
  $z_{0} = \infty$ (resp. $z_{\infty} = \infty$) and we let $p_{0}(t)
  = 1/t$ (resp. $p_{\infty}(t) = 1/t$).

  Let $f := p_{0}(t)/p_{\infty}(t)$. The rational functions
  $f\varphi_m$ and $f^{-1}\varphi_m$ belong to $A_{m}$. By
  Proposition~\ref{sec:LFIHD}~$(a)$ we have
  \[
  f\varphi_{m}\chi^{m}\cdot f^{-1}\varphi_{m}\chi^{m} =
  \varphi_{2m}\chi^{2m}\in\ker\partial,\quad\mbox{and so}\quad 
  f\varphi_{m}\chi^{m},f^{-1}\varphi_{m}\chi^{m}\in\ker\partial\,,
  \]
  yielding a contradiction with Lemma~\ref{lem:5.2}~$(i)$. Hence,
  $\divi\varphi_{m} + \DD(m)$ is supported in at most one point.

  $(iii)$ By $(i)$ and $(ii)$, the map $m\mapsto\DD(m)$ in the
  non-elliptic case, and the map
  $m\mapsto\DD|_{\PP^{1}_{\KK}\setminus\{z_{\infty}\}}(m)$ in the
  elliptic case, are linear in the cone $\omega$. This implies that
  there exists a maximal cone $\omega_{0}$ belonging to $\Lambda(\DD)$
  in the non-elliptic case, and belonging to
  $\Lambda(\DD|_{\PP^{1}_{\KK}\setminus \{z_{\infty}\}})$ in the
  elliptic case, such that $\omega\subseteq\omega_{0}$.

  Let us show the reverse inclusion $\omega_0\subseteq \omega$. Let
  $m\in\omega_{0}$.  Changing $m$ by an integral multiple, we may
  assume $m\in L$ and $\DD(m)$ integral.  By Lemma~\ref{lem:5.2}~$(i)$
  and Proposition~\ref{sec:LFIHD}~$(c)$, the cone $\omega$ is full
  dimensional in $M_{\RR}$. Hence, there exists $m'\in\omega_L$ such
  that $m+m'\in\omega_{L}$. Consider a nonzero section $f_{m}\in
  A_{m}$ such that
  \[
  \divi f_{m} + \DD(m) = 0
  \]
  in the non-elliptic case, and such that
  \[
  \left(\divi f_{m} +
    \DD(m)\right)|_{\PP^{1}_{\KK}\setminus\{z_{\infty}\}} = 0
  \]
  in the elliptic case. It follows that
  \[
  f_{m}\chi^{m}\cdot \varphi_{m'}\chi^{m'} =
  \lambda\varphi_{m+m'}\chi^{m+m'}
  \]
  for some $\lambda\in\KK^{*}$. Therefore,
  $f_{m}\chi^{m}\in\ker\partial$ and again by
  Proposition~\ref{sec:LFIHD}~$(a)$ we have $m\in\omega$.

  $(iv)$ According to the fact that $\sigma^{\vee}_{M}$ spans $M$ and
  that $\partial$ is a homogeneous LFIHD on $A$, for any $m\in M$ we
  have $m + se\in L$ for some $s\in \ZZ$. Changing $r := -s$ by the
  remainder of the euclidian division of $r$ by $d$, if necessary, we
  obtain $m = l + re$, where $l\in L$ and $0\leq r < d$. The
  minimality of $d$ implies that this latter decomposition is unique.

  $(v)$ Assume that $\KK$ is perfect and fix $\bar{\KK}$ an algebraic
  closure of $\KK$.  Consider the algebra $B =
  A\otimes_{\KK}\bar{\KK}$. If we let $\DD = \sum_{z\in
    C}\Delta_{z}\cdot z$, then by Lemma~\ref{lem1.8} the polyhedral
  divisor
  \[
  \DD_{\bar{\KK}} = \sum_{z\in C}\Delta_{z}\cdot S^{\star}(z)
  \]
  over $\PP^{1}_{\bar{\KK}}$ satisfies
  \[
  B =
  \bigoplus_{m\in\sigma^{\vee}_{M}}B_{m}\chi^{m},\quad\mbox{where}\quad
  B_{m} = H^{0}(\PP^{1}_{\bar{\KK}},
  \OO_{\PP^{1}_{\bar{\KK}}}(\DD_{\bar{\KK}}(m)))\,.
  \]
  We can also extend $\partial$ to a homogeneous LFIHD
  $\partial_{\bar{\KK}}$ of horizontal type on $B$. For any
  $m\in\omega_{L}$ we have
  $\varphi_{m}\chi^{m}\in\ker \partial_{\bar{\KK}}$ and there exists a
  rational non-negative number $\lambda_{m}$ such that
  \[
  \divi \varphi_{m} + \DD(m) = \lambda_{m}\cdot z_{\infty}.
  \]
  Applying $S^{\star}$ to the previous equality we obtain
  \[
  \divi_{\bar{\KK}} \varphi_{m} + \DD_{\bar{\KK}}(m) =
  \lambda_{m}\cdot S^{\star}(z_{\infty}).
  \]
  Assume that $z_{\infty}$ is not a rational point and that
  $\lambda_{m}>0$ for some lattice vector $m\in\omega_{L}$. Changing
  $m$ by a multiple we may suppose that $\lambda_{m}$ is greater than
  $1$.  Since the field extension $\bar{\KK}/\KK$ is separable, the
  polynomial $p_{z_{\infty}}(t)$ in the proof of $(ii)$ has at least
  two distinct roots, say $z_{1},z_{2}\in\bar{\KK}$. Note that the
  points $z_{1}, z_{2}$ belong to the support of
  $S^{\star}(z_{\infty})$.  Considering the non-constant rational
  function
  \[
  f = (t-z_{1})/(t-z_{2}),
  \]
  we fall again into a contradiction with Lemma~\ref{lem:5.2}~$(i)$ since
  \[
  f\varphi_{m}\chi^{m}\cdot f^{-1}\varphi_{m}\chi^{m} =
  \varphi_{2m}\chi^{2m}\in \ker \partial_{\bar{\KK}},\quad\mbox{and
    so}\quad f\varphi_{m}\chi^{m}, f^{-1}\varphi_{m}\chi^{m}\in
  \ker \partial_{\bar{\KK}}\,.
  \]
\end{proof}

In the sequel, we let the notation be as in \ref{sit:5.3}. Without
loss of generality, whenever $\KK$ is perfect, in the elliptic case we
can assume that $z_{\infty}$ is the rational point $\infty$ for the
local parameter $t$.

\begin{lemme}\label{lm:5.5}
  Let $\KK$ be a perfect field. The following statements hold.
  \begin{enumerate}[$(i)$]

  \item If $C = \PP^{1}_{\KK}$, then the normalization of the
    subalgebra $A[t]\subseteq\KK(t)[M]$ is $A'=
    A[\AF^{1}_{\KK},\DD|_{\AF^{1}_{\KK}}]$, where $\AF^{1}_{\KK} =
    \spec\KK[t]$.

  \item If the degree of $\partial$ belongs to $\omega$ and the
    evaluation of the polyhedral divisor $\DD|_{\AF^{1}_{\KK}}$ is
    linear, then $\partial$ extends to a homogeneous LFIHD $\partial'$
    on $A'$ of horizontal type.  Furthermore, we have $\ker\partial =
    \ker \partial'$.

  \item Let $d$ be the smallest positive integer such that for any
    $m\in\omega_{M}$ the divisor $\DD(d\cdot m)$ is integral. Then we
    have $d\cdot M\subseteq L$.
  \end{enumerate}
\end{lemme}

\begin{proof}
  $(i)$ This follows from \cite[Theorem~2.5]{La2}.

  $(ii)$ Letting
  \[
  A' =
  \bigoplus_{m\in\sigma^{\vee}_{M}}A'_{m}\chi^{m},\quad\mbox{where}\quad
  A'_{m} = H^{0}(\AF^{1}_{\KK}, \OO_{\AF^{1}_{\KK}}
  (\DD|_{\AF^{1}_{\KK}}(m)))\,,
  \]
  for any $m\in\sigma^{\vee}_{M}$ we can write $A'_{m} =
  \varphi_{m}\cdot\KK[t]$ with $\varphi_{m}$ is a nonzero rational
  function satisfying
  \[
  \divi\varphi_{m} +\lfloor\DD|_{\AF^{1}_{\KK}}(m)\rfloor =0\,.
  \]
  If $m\in\omega_{L}$, we can assume that $\varphi_{m}$ is as in
  Lemma~\ref{lm:5.4}~$(ii)$.

  By Lemma~\ref{sec:mult-syst}, we may extend $\partial$ to a
  homogeneous iterative higher derivation $\partial'$ on the semigroup
  algebra $\KK(t)[M]$.  Denote by $\partial'^{(i)}$ the $i$-th term of
  $\partial'$.  Consider $f\in A'_{m}$ for a lattice vector
  $m\in\sigma^{\vee}_{M}$ and fix an integer $i\in\ZZ_{>0}$.  We will
  show that $\partial'^{(i)}(f\chi^{m})\in A'$.

  By the properness of $\DD$ and Lemma~\ref{lm:5.4}~$(ii)$ with
  $z_\infty=\infty$, we can find a lattice vector $m'\in\omega_{L}$
  verifying the following.  The vectors $m,m'$ belong to a same
  maximal cone of $\Lambda(\DD)$ and the coefficient in $\infty$ of
  the divisor $\divi\varphi_{m'} + \DD(m')$ is integral, positive, and
  greater than that of $-\divi f - \lfloor\DD(m)\rfloor$.  Therefore
  \[
  \divi f\varphi_{m'} + \lfloor\DD(m'+m)\rfloor = \divi f +
  \lfloor\DD(m)\rfloor + \divi\varphi_{m'} + \DD(m') \geq 0.
  \]
  In particular, $\varphi_{m'}f$ belongs to $A_{m + m'}$. Hence it
  follows that
  \[
  \varphi_{m'}\chi^{m'}\partial'^{(i)}(f\chi^{m})
  = \partial^{(i)}(\varphi_{m'}f\chi^{m'+m}) \in A_{m' + m +
    ie}\chi^{m' + m + ie}.
  \]
  By our assumption we have $e\in\omega = \sigma^{\vee}$ so that $m +
  ie\in\sigma^{\vee}_{M}$.  Since $\DD|_{\AF^{1}_{\KK}}$ is linear and
  $\DD(m')$ is integral, we obtain the following identities of
  $\QQ$-divisors over $\AF^{1}_{\KK}$:
  \[
  -\divi \varphi_{m'+m+ie} = \lfloor \DD|_{\AF^{1}_{\KK}} (m' + m +
  ie)\rfloor = \lfloor\DD|_{\AF^{1}_{\KK}}(m')\rfloor + \lfloor
  \DD|_{\AF^{1}_{\KK}}(m+ie)\rfloor\,.
  \]
  Hence,
  \[
  \varphi_{m'+m+ie} = \lambda
  \varphi_{m'}\cdot\varphi_{m+ie}\quad\mbox{for some
  }\lambda\in\KK^{*}.
  \]
  Consequently, this implies
  \[
  \varphi_{m'}\chi^{m'}\partial'^{(i)}(f\chi^{m})\in
  A_{m'+m+ie}\chi^{m'+m+ie} \subseteq
  \varphi_{m'}\cdot\varphi_{m+ie}\cdot\KK[t]\,\chi^{m'+m+ie}\,.
  \]
  This yields
  \[
  \partial'^{(i)}(f\chi^{m})\in\varphi_{m+ie}\cdot\KK[t]\,\chi^{m+ie}
  = A'_{m+ie}\chi^{m+ie}\subseteq A',
  \]
  as required. We conclude that the subalgebra $A'$ is
  $\partial'$-invariant.

  Next, we show that $\partial'$ is a homogeneous LFIHD on $A'$.  Let
  $m'$ be as above.  We have $t\varphi_{m'}\chi^{m'}\in A$.  Thus,
  there exists $s\in\ZZ_{>0}$ such that
  \[
  \varphi_{m'}\chi^{m'}\partial'^{(i)}(t)
  = \partial^{(i)}(t\varphi_{m'}\chi^{m'}) = 0\quad\mbox{for any}\quad
  i\geq s.
  \]
  Hence $\partial'$ acts locally finitely on $t$ and so the same holds
  for $A[t]$. Let $f\in A'_{m}$ and choose $s'\in\ZZ_{>0}$ such that
  the sheaf $\OO_{\PP^{1}_{\KK}}(\lfloor \DD(m+s'm')\rfloor)$ is
  globally generated. Thus,
  \[
  \varphi_{s'm'}f\chi^{m+s'm'}\in A'_{m+s'm'} =
  \KK[t]\otimes_{\KK}A_{m+s'm'}\subseteq A[t]\,.
  \]
  Since $\varphi_{s'm'}\chi^{s'm'}$ is in the kernel of $\partial$ we
  conclude that $\partial'$ acts locally finitely on $f\chi^{m}$.
  This proves that $\partial'$ is an LFIHD. The fact that $\partial'$
  is of horizontal type is straightforward and the proof is left to
  the reader.

  It remains to show that $\ker \partial = \ker \partial'$.  By
  Lemma~\ref{lem:5.2}~$(i)$ the kernel $\ker\partial'$ is the
  semigroup algebra given by $\omega_{L'}$, where $L'$ is a sublattice
  of maximal rank. Since $\ker\partial\subseteq \ker \partial'$ we
  have $L\subseteq L'$ and $L'/L$ is a finite abelian group. Let
  \[
  \ker\partial
  =\bigoplus_{m\in\omega_{L}}\KK\cdot\varphi_{m}\chi^{m}\quad\mbox{and}\quad
  \ker\partial' =\bigoplus_{m\in\omega_{L'}} \KK\cdot\varphi'_{m}\chi^{m}\,.
  \]
  Letting $m\in L'$ we let $r\in\ZZ_{>0}$ be such that $rm\in
  L$. Then, by Lemma~\ref{lm:5.4}~$(i)$ and $(ii)$ we can write
  $\lambda\varphi_{rm} = \varphi'_{rm} = (\varphi'_{m})^{r}$, where
  $\lambda\in\KK^{*}$. So $\varphi'_{m}\chi^{m}$ is integral over
  $\ker\partial$. By normality of $A$ and since $\ker\partial$ is
  algebraically closed in $A$ one has $\varphi'_{m}\chi^{m}\in\rm
  ker\,\partial$. Hence $L' = L$ and so $\ker\partial =\ker\partial'$.

  $(iii)$ Up to multiplying the LFIHD $\partial$ by a homogeneous
  kernel element, we may assume that $\deg\partial=e\in \omega$.  In
  particular, the algebra
  \[
  A_{\omega} = \bigoplus_{m\in\omega_{M}}A_{m}\chi^{m}\quad\mbox{is }
  \partial\mbox{-invariant.} \]

  By virtue of assertions $(i)$ and $(ii)$ in the lemma, we may
  suppose that $C=\AF_\KK^1$.  Let $m\in\omega_{M}$. We have $
  A_{dm+m'} = A_{dm}\cdot A_{m'} = \varphi_{dm} A_{m'}$ for all
  $m'\in\omega_{M}$. Hence, the principal ideal
  $(\varphi_{dm}\chi^{dm})$ in the ring $A_{\omega}$ is
  $\partial|_{A_{\omega}}$-invariant.  By
  Proposition~\ref{sec:LFIHD}~$(f)$, we have
  $\varphi_{dm}\chi^{dm}\in\ker\partial$ and so
  $dm\in\omega_{L}$. This yields $d\cdot\omega_{M}\subseteq\omega_{L}$
  and $(iii)$ follows.
\end{proof}

The following result provides a geometrical characterization of
horizontal non-hyperbolic affine $\GM$-surfaces. See \cite[Theorem~3.3
and 3.16]{FZ2} for the case where the base field is $\mathbb{C}$.

\begin{corollaire} \label{cor:5.6}
  Assume $\KK$ is perfect. Let $N = \ZZ$ and $\sigma = \RR_{\geq 0}$,
  so that $\DD$ is uniquely determined by the $\QQ$-divisor $\DD(1)$.
  If the graded algebra $A$ admits a homogeneous LFIHD of horizontal
  type, then the following statements hold.
  \begin{enumerate}[$(i)$]
  \item If $C = \AF^{1}_{\KK}$, then the fractional part $\{\DD(1)\}$
    has at most one point in its support.
  \item If $C = \PP^{1}_{\KK}$, then $\{\DD(1)\}$ has at most two
    points in its support.
  \end{enumerate}
  In each case, the support of $\{\DD(1)\}$ consists of rational
  points. In particular, every horizontal non-hyperbolic affine
  $\GM$-surface over $\KK$ is toric.
\end{corollaire}

\begin{proof}
  $(i)$ We first prove the result in the case where $\KK$ is
  algebraically closed. Let $d$ be the smallest positive integer such
  that $\DD(d)$ is an integral divisor.  Letting $f\in \KK(t)$ a
  generator of $A_{d}$, i.e. $A_{d} = f\cdot A_{0}$, we let $B$ be the
  integral closure of $A[\sqrt[d]{f}\chi]$ in its field of fractions.
  Up to a principal divisor, we may assume $\DD(1)<0$ and so $f\in
  \KK[t]$ is a polynomial.  By Lemma~\ref{lm:5.5}~$(ii)$, we have
  $f\chi^{d}\in\ker\partial$.

  By Corollary~\ref{cor:2.6}, we obtain the existence of an LFIHD
  $\partial'$ on $B$ extending $\partial$ and satisfying
  $\sqrt[d]{f}\chi\in\ker\partial'$.  Write $B = A[C',\DD']$ for some
  polyhedral divisor $\DD'$ on a regular affine curve $C' = \spec
  B_{0}$.  Actually, $B_{0}$ is the normalization of
  $\KK[t,\sqrt[d]{f}]$ and also a polynomial algebra of one variable
  over $\KK$ (see Lemma~\ref{lem:5.2}~$(iii)$). The fact that
  $B_{0}^{*} = \KK^{*}$ and that $B_{0}$ is an unique factorization
  domain implies that $f = (t-z)^{r}$ for some $z\in \KK$ and some
  $r\in\ZZ_{>0}$.  Since $\divi f + d\cdot\DD(1) = 0$ one concludes
  that $\{\DD(1)\}$ is supported in at most on the point $z$.

  Assume now that $\KK$ is not algebraically closed and that
  $\{\DD(1)\}$ is supported in at least two points. Extending the
  scalar to the algebraic closure $\bar{\KK}$ gives a contradiction by
  Lemma~\ref{lem1.8}.

  \medskip
  
  $(ii)$ Multiplying $\partial$ by a homogeneous element in its
  kernel, we may assume that the degree of $\partial$ is
  non-negative. By Lemma~\ref{lm:5.5}~$(ii)$, the LFIHD $\partial$
  extends to a homogeneous LFIHD $\partial'$ of horizontal type on the
  normalization $A'$ of the algebra $A[t]$. Note that the graded
  algebra $A'$ is given by the polyhedral divisor
  $\DD|_{\AF^{1}_{\KK}}$.  Applying $(i)$ for the non-elliptic graded
  algebra $A'$, the fractional part $\{\DD_{|\AF^{1}_{\KK}}(1)\}$ has
  at most one point in its support. So $\{\DD(1)\}$ is supported in at
  most two points. This yields $(ii)$.

  Let us show the latter claim.  By a similar argument, we deduce that
  in any case the support of $\{\DD(1)\}$ consists of rational points
  (see Lemma~\ref{lem1.8}).  Assume that $A$ is non-elliptic. Since
  $\{\DD(1)\}$ is supported in at most one rational point, without
  loss of generality, we can let
  \[
  \DD(1) = -\frac{e}{d}\cdot 0,\quad\mbox{where}\quad 0\leq e < d,
  \quad\mbox{and}\quad \gcd(e,d) = 1\,.
  \]
  A straightforward computation shows that
  \[
  A = \bigoplus_{b\geq 0,\,ad-be\geq 0}\KK\,t^{a}\chi^{b},
  \]
  see e.g. \cite[Lemma~3.8]{FZ2} and \cite[Example~3.20]{Li}. The
  algebra $A$ admits an effective $\ZZ^{2}$-grading endowing $X =
  \spec A$ with a structure of a toric surface. Assume that $A$ is
  elliptic.  Using the fact that every integral divisor over $\PP^{1}$
  of degree $0$ is principal, we can reduce to the case where $\DD$ is
  supported in the points $0$ and $\infty$.  We conclude by a similar
  argument as in \cite[Example~3.21]{Li}.
\end{proof}

As a consequence of Corollary~\ref{cor:5.6}, we obtain the following
result.

\begin{corollaire}
  With the notation in \ref{sit:5.3}, we let $A_{\omega} =
  \bigoplus_{m\in\omega_{M}}A_{m}\chi^{m}$ and let $\tau =
  \omega^{\vee}\subseteq N_\RR$.  Then $A_{\omega}\simeq
  A[C,\DD_{\omega}]$ as $M$-graded algebras, where $\DD_{\omega}$ is
  $\tau$-proper polyhedral divisor over the curve $C$ satisfying the
  following conditions.
  \begin{enumerate}[$(i)$]
  \item If $A$ is non-elliptic, then $\DD_{\omega} =
    (v+\tau)\cdot 0$ for some $v\in N_{\QQ}$.
  \item If $A$ is elliptic, then $\DD_{\omega} = (v+\tau)\cdot 0 +
    \Delta'_{\infty}\cdot\infty$ for some $v\in N_{\QQ}$ and some
    $\Delta'_{\infty}\in \pol_{\tau}(N_{\RR})$ satisfying $v +
    \Delta'_{\infty}\subsetneq \tau$.
  \end{enumerate}
\end{corollaire}

\begin{proof}
  $(i)$ We will follow the argument in \cite[Lemma~3.23]{Li}. Note
  that the degree $e$ of $\partial$ belongs to $\omega$. For
  $\ell\in\omega_{L}$ denote by $\partial_{\ell}$ the homogeneous
  LFIHD $\varphi_\ell\cdot\partial$.  The subalgebra
  \[
  B_{(\ell+e)} = \bigoplus_{r\geq 0}A_{r(\ell+e)}\chi^{r(\ell+e)}
  \]
  is $\partial_{\ell}$-invariant. Since the homogeneous LFIHD
  $\partial_{\ell}|_{B_{(\ell+e)}}$ is of horizontal type, we can
  apply Corollary~\ref{cor:5.6} to conclude that $\{\DD(\ell+e)\}$ is
  supported in at most one point. By Lemma~\ref{lm:5.4}~$(i)$, for all
  $\ell,\ell'\in\omega_{L}$ we have
  \[
  -\divi\varphi_{\ell'} + \DD(\ell+e) = \DD(\ell+\ell'+e) =
  \DD(\ell'+e) - \divi\varphi_{\ell}\,,\quad\mbox{and so}\quad
  \{\DD(\ell+e)\} = \{\DD(\ell'+e)\}\,.
  \]

  Thus, the union of the supports of the divisors $\{\DD(\ell+e)\}$
  has at most one element, where $\ell$ runs over $\omega_{L}$. By the
  linearity of $\DD$ in $\omega$ and Lemma~\ref{lm:5.4}~$(iv)$, up to
  a principal polyhedral divisor, the polyhedral divisor $\DD_\omega$
  of $A_{\omega}$ is supported in at most one point. This point needs
  to be rational so $(i)$ follows.

  $(ii)$ By multiplying $\partial$ with a kernel element, we may
  assume $e\in\omega$. Let $A'_{\omega}$ be the normalization of
  $A_{\omega}[t]$. By Lemma~\ref{lm:5.5}, elements of degree
  $m\in\omega_{M}$ in $A'_{\omega}$ correspond to the product of a
  global section of $\DD|_{\AF^{1}_{\KK}}(m)$ and the character
  $\chi^{m}$. In addition, $\partial$ extends to a homogeneous LFIHD
  of horizontal type on $A'_{\omega}$. By $(i)$, the union of the
  supports of the divisors $\{\DD|_{\AF^{1}_{\KK}}(m)\}$, where $m$
  runs trough $\omega_{M}$, has at most one rational point.  This
  concludes the proof.
\end{proof}

For our next theorem, which is a key ingredient in our classification
result, we introduce the following notation. Let $\DD$ be a proper
$\sigma$-polyhedral divisor over $\AF^{1}_{\KK}$ or $\PP^{1}_{\KK}$
such that the coefficient $\Delta_0$ at zero is $v+\sigma$ for some
$v\in N_\QQ$. Let $\widehat{M}=M\times \ZZ$ and let
$\widehat{N}=N\times \ZZ$. We also let $\widehat{\sigma}$ be the cone
in $\widehat{N}_\RR$ generated by $(v,1)$ and $(\sigma,0)$ if
$C=\AF_\KK^1$ and by $(v,1)$,$(\sigma,0)$ and $(\Delta_\infty,-1)$ if
$C=\PP_\KK^1$.

\begin{theorem}\label{th:5.8}
  Let $\DD$ be a $\sigma$-proper polyhedral divisor over a regular
  curve $C$.  Assume that $\DD$ satisfies one of the following
  conditions.
  \begin{enumerate}[(i)]
  \item If $C$ is affine, then $C = \AF^{1}_{\KK} = \spec\KK[t]$ and
    $\DD = (v + \sigma)\cdot 0$ for some $v\in N_{\QQ}$.
  \item If $C$ is projective, then $C = \PP^{1}_{\KK}$ and $\DD = (v +
    \sigma)\cdot 0 + \Delta_{\infty}\cdot\infty$ for some $v\in
    N_{\QQ}$ and for some $\Delta_{\infty}\in
    \operatorname{Pol}_{\sigma}(N_{\RR})$.
  \end{enumerate}
  
  Let $d$ be the smallest positive integer such that $dv\in N$. For
  any $m\in M$ we let $h(m) = \langle m, v\rangle$.  Then there exists
  a homogeneous LFIHD $\partial$ of horizontal type on $A = A[C,\DD]$
  with $\deg\partial = e$ if and only if the following statements
  hold.
  \begin{enumerate}[(a)]
  \item If $\chara\KK=p>0$, then there exists a sequence of integers
    $0\leq s_{1} < s_{2} < \ldots < s_{r}$ such that for $i =
    1,\ldots, r$ we have $\big(p^{s_i}e,-1/d -
    h(p^{s_i}e)\big)\in\rt\widehat{\sigma}$.
  \item If $\chara\KK=0$, then $\big(e,-1/d -
    h(e)\big)\in\rt\widehat{\sigma}$.
  \end{enumerate}
  Under these latter conditions, the LFIHD $\partial$ is of following
  form. Let $\zeta = \sqrt[d]{t}$.  Let us consider the LFIHD
  $\partial_{\zeta}$ on the algebra $\KK[\zeta]$ with exponential map
  \begin{align} \label{eq:6} %
    e^{x\partial_{\zeta}}(\zeta) = \zeta + \sum_{i =
      1}^{r}\lambda_{i}x^{p^{s_{i}}},
  \end{align}
  where $\lambda_{1},\ldots, \lambda_{r}\in\KK^{*}$ (resp. with
  $\partial^{(1)}_{\zeta} = \lambda \frac{\rm d}{\rm d\it \zeta}$,
  where $\lambda\in\KK^{*}$) whenever $\rm char\it\,\KK\rm >0$
  (resp. $\rm char\it\, \KK\rm = 0$).  Then the $i$-th term of
  $\partial$ is given by the equality
  \begin{align} \label{eq:7}
    \partial^{(i)}(t^{l}\chi^{m}) =
    \zeta^{-dh(m+ie)}\partial^{(i)}_{\zeta}(\zeta^{dh(m)}t^{l})\chi^{m+ie}
    \quad\mbox{for all}\quad t^{l}\chi^{m}\in A\,.
  \end{align}
\end{theorem}

\begin{proof}
  Assume that $\DD$ satisfies $(i)$ and fix an LFIHD $\partial$ on the
  algebra $A$ of horizontal type and of degree $e$. Let $B$ be the
  normalization of the subalgebra
  \[
  A\left[\zeta^{-dh(e)}\chi^{e}\right]\subseteq\KK(\zeta)[M].
  \]
  Consider the affine line $C' = \spec\KK[\zeta]$ and the polyhedral
  divisor $\DD' = (dv + \sigma)\cdot 0$ over $C'$. Since $d=\min\{r\in
  \ZZ_{>0}\mid re\in L\}$ (see Lemma~\ref{lm:5.4}~$(iv)$), the algebra
  $A[C',\DD']$ is precisely $B$ (see \cite[Theorem
  2.5]{La2}). According to Lemma~\ref{4.1}~$(ii)$ we have
  $e\in\sigma^{\vee}$ and so $A\left[\zeta^{-dh(e)}\chi^{e}\right]$ is
  a cyclic extension of the ring $A$. Since
  $\varphi_{de}\chi^{de}\in\ker\partial$ by Corollary~\ref{cor:2.6},
  $\partial$ extends to a unique LFIHD $\partial'$ on $B$. Using
  further that $dv\in N$ we obtain a natural isomorphism of $M$-graded
  algebras
  \[
  \varphi: B\rightarrow E, \qquad \zeta^{l}\chi^{m}\mapsto
  \zeta^{dh(m)+l} \chi^{m},
  \]
  where $E = \KK[\sigma^{\vee}_{M}][\zeta]$.  Consider
  $\varphi_{*}\partial'$ the homogeneous LFIHD of horizontal type on
  $E$ given by
  \[
  \varphi_{*}\partial'^{(i)} =
  \varphi\circ\partial'^{(i)}\circ\varphi^{-1},
  \]
  where $i\in\ZZ_{\geq 0}$.  Now, Lemma~\ref{lm:5.5}~$(iii)$ implies
  that $\ker\varphi_{*}\partial' = \KK[\sigma^{\vee}_{M}]$ so that
  $\varphi_{*}\partial' = \chi^{e}\cdot\partial_{\zeta}$ for some
  non-trivial LFIHD $\partial_{\zeta}$. An easy computation shows that
  the LFIHD $\partial = \varphi_{*}^{-1}(\varphi_{*}\partial')$ is as
  in \eqref{eq:7}.
  
  Assume that $\chara\KK=p>0$ and let us show that $(a)$ holds. By
  Proposition~\ref{sec:LFIHD}~$(d)$, the exponential map of
  $\partial_{\zeta}$ is given as in \eqref{eq:6} for some integers
  $0\leq s_{1}<\ldots < s_{r}$.  If $p$ does not divide $d$, then
  consider $l\in\ZZ_{\geq 0}\setminus p\ZZ$ such that $dl\geq
  p^{s_{1}}$. Note that $t^{l}\in A$. By Lemma~\ref{sec:line} and
  \eqref{eq:7} we obtain the equality
  \[
  \partial^{(p^{s_{1}})}(t^{l}) =
  \lambda_1dlt^{-1/d-h(p^{s_{1}}e)+l}\chi^{p^{s_{1}}e}\,.
  \]
  Since $\partial^{(p^{s_{1}})}(t^{l})\in A\setminus\{0\}$, it follows that
  $-1/d-h(p^{s_{1}}e)\in\ZZ$.

  Otherwise, assume that $p$ divide $d$. By the minimality of $d$
  there exists $m\in\sigma^{\vee}_{M}$ such that $dh(m)$ is not
  divisible by $p$. Taking $l\in\ZZ_{\geq 0}$ such that $dl\geq
  \max\{p^{s_{1}},-dh(m)\}$ we have $t^{l}\chi^{m}\in A\setminus\{0\}$
  and so Lemma~\ref{sec:line} implies
  \[
  \partial^{(p^{s_{1}})}(t^{l}\chi^{m}) =
  \lambda_1dh(m)t^{-1/d-h(p^{s_{1}}e)+l}\chi^{m+p^{s_{1}}e}\in A\setminus\{0\}\,.
  \]
  Hence in any case $e_{1} : = (p^{s_{1}}e, -1/d-h(p^{s_{1}}e))\in
  \widehat{M}$, where $\widehat{M} = M\times\ZZ$.

  Let us remark that
  \[
  A[C,\DD] =
  \bigoplus_{(m,l)\in\widehat{\sigma}^{\vee}_{\widehat{M}}}\KK\,\chi^{(m,l)}
  = \KK[\widehat{\sigma}^{\vee}_{\widehat{M}}],
  \]
  where $\chi^{(m,l)} = t^{l}\chi^{m}$ and $\widehat{\sigma}$ is the
  cone generated by $(v,1)$ and $(\sigma,0)$. Since $e\in
  \sigma^\vee$, an easy computation shows that $e_{1} = (p^{s_{1}}e,
  -1/d-h(p^{s_{1}}e))\in \rt\widehat{\sigma}$ for the distinguished
  ray $\rho = (dv,d)$. So by Corollary~\ref{cor3.6} the
  $\widehat{M}$-graded algebra $A$ admits rationally homogeneous
  LFIHDs of degree $e_{1}/p^{s_{1}}$ coming from the root $e_{1}$. One
  of such rationally homogeneous LFIHDs is given by the equality
  \[
  e^{x\partial_{1}}(t^{l}\chi^{m}) = \sum_{i =
    0}^{\infty}\binom{d(l+h(m))}{i}
  \lambda_{1}^{i}t^{l-i(1/d+h(p^{s_{1}}e))}\chi^{m+ip^{s_{1}}e}x^{ip^{s_{1}}},
  \]
  where $\lambda_{1}\in\KK^{*}$ is as \eqref{eq:6}. Furthermore, by
  Corollary~\ref{cor:2.6} we extend $\partial_{1}$ to a homogeneous
  LFIHD $\partial_{1}'$ on the $M$-graded algebra $B$.  Assume that
  $r\geq 2$. One can see $e^{x\partial'}$ and $e^{x\partial_{1}'}$ as
  automorphisms of the algebra $B[x]$ by letting $e^{x\partial'}(x) =
  e^{x\partial'_{1}}(x) = x$.  Hence, using this convention we have
  \[
  e^{x\partial'}\circ (e^{x\partial'_{1}})^{-1} =
  e^{x\varphi_{*}^{-1}(\chi^{e}\partial_{\zeta,1})},
  \]
  where $\partial_{\zeta,1}$ is the LFIHD on $\KK[\zeta]$ defined by
  \[
  e^{x\partial_{\zeta,1}}(\zeta) = \zeta + \sum_{i =
    2}^{r}\lambda_{i}x^{p^{s_{i}}}\,.
  \]
  Consequently, the map
  $e^{x\partial'}\circ (e^{x\partial'_{1}})^{-1}$ yields a homogeneous
  LFIHD $\partial''_{1}$ on $A$. Actually, replacing
  $\partial_{\zeta,1}$ by $\partial_{\zeta}$, the LIFHD
  $\partial''_{1}$ satisfies \eqref{eq:7}. Again, it follows that
  $e_{2}:= (p^{s_{2}}e, -1/d-h(p^{s_{2}}e))\in \hat{M}$ is a root of
  $\widehat{\sigma}$. One concludes by induction that $(a)$ holds.
 
  If $\chara\KK=0$, then the locally nilpotent derivation
  $\partial^{(1)}_{\zeta}$ on the algebra $\KK[\zeta]$ is equal to
  $\lambda \frac{\partial}{\partial \zeta}$ for some
  $\lambda\in\KK^{*}$.  Using \eqref{eq:7} we have
  \[
  \partial^{(1)}(t) = \lambda dt^{-1/d-h(e) + 1}\chi^{e}\in
  A\setminus\{0\}
  \]
  and so assertion $(b)$ holds. This concludes the proof in the case
  where condition $(i)$ holds.

  Assume now that $(ii)$ holds. Let $A'$ be the normalization of
  $A[t]$ in the field $\fract A$.  By Lemma~\ref{lm:5.5}~$(iii)$ , we
  have $ d\cdot M = h^{-1}(\ZZ) \subseteq L$, where $L$ is the
  sublattice of $M$ generated by the set of weights of
  $\ker \partial$.  Hence, changing $\partial$ by
  $\varphi_{m}\cdot\partial$ for $m\in\sigma^{\vee}_{d\cdot M}$,
  without loss of generality, we may assume
  $e\in\sigma^{\vee}_M$.

  More precisely, replacing $e$ by $e + m$ for some
  $m\in\sigma^{\vee}_{d\cdot M}$ does not change assertions $(a),(b)$
  in the Theorem. With this new assumption, again by
  Lemma~\ref{lm:5.5}, we extend $\partial$ to a homogeneous LFIHD
  $\bar{\partial}$ on $A'$ of horizontal type. By the previous
  argument (the case where $C = \AF^{1}_{\KK}$) applied to
  $(A',\bar{\partial})$ and since $\bar\partial$ stabilizes
  $\KK[\widehat{\sigma}^\vee\cap\widehat{M}]$ we obtain $(a)$ and
  $(b)$.

  It remains to show that if a lattice vector $e$ verifies assertions
  $(a),(b)$, then one can build a homogeneous LFIHD on $A = A[C,\DD]$
  of horizontal type and of degree $e$ as in \eqref{eq:7}. Assume that
  $\chara\KK>0$ and let $e_{i} = (e,-1/d-h(p^{s_{i}}e))$. By $(a)$ we
  have $e_{i}\in\rt \widehat{\sigma}$ and we can consider the
  rationally homogeneous LFIHDs
  $\partial_{e_{1},s_{1}},\ldots, \partial_{e_{r},s_{r}}$ on the
  semigroup algebra $\KK[\widehat{\sigma}^{\vee}_{\widehat{M}}]$ (see
  Example~\ref{sec:ex3.2}). Using the isomorphism $\varphi$ and
  considering every $e^{x\partial_{e_{i},s_{i}}}$ as automorphism of
  the ring $A[x]$, a computation shows that the composition
  \[
  e^{x\partial_{e_{1},s_{1}}}\circ e^{x\partial_{e_{2},s_{2}}}\circ
  \ldots\circ e^{x\partial_{e_{r},s_{r}}}
  \]
  defines an LFIHD as in \eqref{eq:7}. In the case where $\chara\KK\rm
  =0$, a similar argument can be applied (see also \cite[Examples 3.20
  and 3.21]{Li}). We leave the details to the reader.
\end{proof}

For the proof of our next lemma, which is the last ingredient for our
main theorem, we need the following remark.

\begin{remarque} \label{rm:5.9} %
  Assume that $\KK$ is perfect and let $r\in \ZZ_{>0}$. Then the
  Frobenius map $F:\KK\rightarrow \KK$ mapping $\lambda\mapsto
  \lambda^{p^r}$ is a field automorphism. Let $t$ be a new variable
  and let $x=t^{p^r}$. We will compute the ramification of the field
  extension $\KK(t)/\KK(x)$. Let $P(x)=\sum a_ix^i\in \KK[x]$ be an
  irreducible polynomial. Then
  \[P(x)=P(t^{p^r})=\left(F^*(P)(t)\right)^{p^r},
  \quad\mbox{where}\quad F^*(P)(t)=\sum F^{-1}(a_i)t^i\,.\]

  Hence $F^*(P)(t)$ is irreducible in $\KK[t]$. Let $C$ and $C'$ be
  unique projective curves over $\KK$ whose function fields are
  $\KK(t)$ and $\KK(x)$, respectively (both isomorphic to
  $\PP^1_\KK)$. The inclusion $\KK(x)\subseteq\KK(t)$ induces a purely
  inseparable morphism $\pi:C\rightarrow C'$. Our previous computation
  shows that for every $z\in C$ the pullback of $z$ as Weil divisor is
  given by $\pi^*(z)=p^r\cdot z'$, where $z'\in C'$ lies in the
  schematic fiber of $z$.
\end{remarque}

Let $\DD=\sum_{z\in C} \Delta_z\cdot z$ be proper $\sigma$-polyhedral
divisor over a regular curve $C$. Recall that $h_z$ stands for the
support function of the $\sigma$-polyhedron $\Delta_z$ for all $z\in
C$, see Definition~\ref{sec:def-convex}.

\begin{lemme}
  Assume that $\KK$ is perfect. Let $\DD$ be a proper
  $\sigma$-polyhedral divisor over $C = \AF^{1}_{\KK}$ or $C =
  \PP^{1}_{\KK}$, respectively. Assume that there exists a maximal
  cone $\omega$ on the quasifan $\Lambda(\DD)$ or
  $\Lambda(\DD_{|\AF^{1}_{\KK}})$, respectively, such that for any
  $z\in C$ different from $0$ and $\infty$ we have $h_{z}|_{\omega} =
  0$. Let $\partial$ be an LFIHD of degree $e$ on the algebra
  $A[C,\DD_{\omega}]$ given by formula~\eqref{eq:7}. Let $p=\chara\KK$
  if $\chara\KK>0$ and $p = 1$ if $\chara\KK=0$. Then $\partial$
  extends to an LFIHD on $A = A[C,\DD]$ if and only if for any
  $m\in\sigma^{\vee}_{M}$ such that $m + p^{s_{1}}e\in
  \sigma^{\vee}_{M}$ the following hold.
\begin{enumerate}[$(i)$]
\item If $h_{z}(m+p^{s_{1}}e)\neq 0$, then $\lfloor
  p^kh_{z}(m+p^{s_{1}}e)\rfloor - \lfloor p^kh_{z}(m)\rfloor\geq 1$,
  $\forall z\in C$, $z\neq 0,\infty$.
\item If $h_{0}(m+p^{s_{1}}e)\neq h(m + p^{s_{1}}e)$, then $\lfloor
  dh_{0}(m+p^{s_{1}}e)\rfloor - \lfloor dh_{0}(m)\rfloor\geq 1 +
  dh(p^{s_{1}}e)$.
\item If $C = \PP^{1}_{\KK}$, then $\lfloor
  dh_{\infty}(m+p^{s_{1}}e)\rfloor - \lfloor dh_{\infty}(m)\rfloor
  \geq -1 - dh(p^{s_{1}}e)$.
\end{enumerate}
Here $h$ is the linear extension of $h_{0}|_{\omega}$ to $M_\RR$,
$d\in\ZZ_{>0}$ is the smallest positive integer such that $dh$ is
integral and $k$ is the unique non-negative integer such that
$d=d'p^k$ with $\gcd(d',p)=1$. 
\end{lemme}

\begin{proof}
  Considering $m\in\sigma^{\vee}_{M}$ we can write $h(m) = \langle m,
  v\rangle$ for some $v\in N_{\QQ}$.  Since every $h_{z}$ is upper
  convex, $h_{z}(m)\leq 0$ $\forall z\in C\setminus\{0,\infty\}$, and
  obviously $h_{0}(m)\leq h(m)$.  Letting
  \[
  A_{M} = \bigoplus_{m\in M}\KK[t]\cdot\varphi_{m}\chi^{m},
  \]
  where $\varphi_{m} = t^{-\lfloor h(m)\rfloor}$ and localizing by a
  homogeneous element of $\ker\partial$, by Lemma~\ref{sec:mult-syst},
  $\partial$ extends to a homogeneous LFHID on $A_{M}$. We also denote
  this extension by $\partial$. Hence, $\partial$ extends to an LFIHD
  on $A$ if and only if the extension $\partial$ on $A_{M}$ stabilizes
  $A$. In addition, we may assume that $\KK=\bar\KK$ is algebraically
  closed since the extension $\partial_{\bar\KK}$ of $\partial$ on
  $A_M\otimes_\KK\bar\KK$ stabilizes $A\otimes_\KK\bar\KK$ if and only
  if $\partial$ stabilizes $A$.

  For the characteristic zero case, the proof is available in
  \cite[Lemma~3.26]{Li}. In the sequel, we assume $\chara\KK = p >0$.
  The proof is divided into three steps, (similar to
  \cite[Lemma~3.26]{Li}) where we assume $h=0$, $h(m)$ integral for
  all $m$ and finishing with the general case.

  \medskip
  
  \emph{Case $h = 0$.}  In this case we have $d = 1$, $L = M$ and by
  Theorem~\ref{th:5.8}, $\partial =\chi^{e}\partial_{t}$ for some
  LFIHD $\partial_{t}$ on $\KK[t]$. By
  Proposition~\ref{sec:LFIHD}~$(d)$, the LFIHD $\partial_{t}$ is
  determined by a sequence of integers $0\leq s_{1}<\ldots < s_{r}$.
  Furthermore, since $h_{z}\leq 0$ for any $z\in\AF^{1}_{\KK}$, then
  $h_{\infty}\geq 0$ in the elliptic case.  Fixing
  $m\in\sigma^{\vee}_{M}$ such that
  $m + p^{s_{1}}e\in\sigma^{\vee}_{M}$ the conditions of our lemma
  become:
  \begin{enumerate}
  \item[$(i')$] If $h_{z}(m+p^{s_{1}}e)\neq 0$, then $\lfloor h_{z}(m +
    p^{s_{1}}e)\rfloor - \lfloor h_{z}(m)\rfloor \geq 1$ $\forall
    z\in\AF^{1}_{\KK}$.
  \item[$(iii')$] If $C = \PP^{1}_{\KK}$, then $\lfloor
    h_{\infty}(m+p^{s_{1}}e)\rfloor - \lfloor h_{\infty}(m)\rfloor
    \geq -1$.
  \end{enumerate}

  Under the above assumption we have
  \[
  A_{m} = H^{0}(C,\OO_{C}(\DD(m)))\subseteq \KK[t]
  \]
  and $\partial$ stabilizes $A$ if and only if
  \[
  f(t)\in A_{m}\Rightarrow \partial^{(i)}_{t}(f(t))\in A_{m + ie},
  \forall m\in\sigma^{\vee}_{M},\quad \forall i\in\ZZ_{\geq 0}\,,
  \]
  or equivalently,
  \[
  \divi f + \lfloor \DD(m)\rfloor \geq 0\Rightarrow
  \divi\partial_{t}^{(i)}(f) + \lfloor\DD(m+ie)\rfloor\geq 0, \quad
  \forall m\in\sigma^{\vee}_{M},\ \forall i\in\ZZ_{\geq 0}\,.
  \]
  This is also equivalent to
  \begin{align}\label{eq:8}
    \ord_{z}f + \lfloor h_{z}(m)\rfloor \geq 0\Rightarrow
    \ord_{z}\partial_{t}^{(i)}(f) + \lfloor h_{z}(m + ie)\rfloor
    \geq 0,\quad \forall m\in\sigma^{\vee}_{M},\ \forall i\in\ZZ_{\geq 0},\ 
    \forall z\in C\,. 
  \end{align}

  \medskip
  
  We will first show the lemma in the case where $C=\AF^1_\KK$. Let us
  show first that $(i')$ implies \eqref{eq:8} and so $\partial$
  stabilizes $A$.  If $h_{z}(m+p^{s_{1}}e)\neq 0$ with
  $m\in\sigma^{\vee}_{M}$ such that
  $m+p^{s_{1}}e\in\sigma^{\vee}_{M}$. Then we have $h_{z}(m)\neq 0$ so
  that $f\in (t-z)\KK[t]$.

  Let $i\in\ZZ_{\geq 0}$. If $\partial^{(i)}_{t}(f) = 0$, then
  $\partial^{(i)}_{t}(f)\in A_{m+ie}$.  Otherwise,
  $\partial^{(i)}_{t}(f)\neq 0$ and so $m+ie\in\sigma^{\vee}$.
  Letting $i = lp^{s_{1}}$ for some $l\in\ZZ_{\geq 0}$, we have
  $\ord_z\partial^{(i)}_{t}(f)\geq \ord_z(f)-l$. Hence it follows that
  \[
  \ord_{z}\partial^{(i)}(f) + \lfloor h_{z}(m + ie)\rfloor \geq
  \ord_{z}(f) + \lfloor h_{z}(m)\rfloor + (\lfloor h_{z}(m +
  lp^{s_{1}}e)\rfloor - \lfloor h_{z}(m)\rfloor -l).
  \]

  By convexity of $\sigma^{\vee}$ for $1\leq j\leq l$ we have $m +
  jp^{s_{1}}e\in \sigma^{\vee}$. If $h_z(m+ie)=0$, then
  $\ord_z\partial^{(i)}(f)+\lfloor h_{z}(m+ie)\rfloor\geq 0$
  and~\eqref{eq:8} holds. Otherwise, $h_z(m+ie)\neq 0$ and again
  $h_z\big(m+(l-j)p^{s_{1}}e\big)\neq 0$ for $1\leq j\leq l$.
  Combining the previous inequality with $(i')$, and the fact that
  $\ord_{z}f + \lfloor h_{z}(m)\rfloor \geq 0$ we obtain
  \begin{align*}
    \ord_{z}\partial^{(i)}(f) + \lfloor h_{z}(m + ie)\rfloor
    \geq &\ord_{z}(f) + \lfloor h_{z}(m)\rfloor + \\
    &\sum_{j = 1}^{l}(\lfloor h_{z}(m + (l-j)p^{s_{1}}e +
    p^{s_{1}}e)\rfloor - \lfloor h_{z}(m + (l-j)p^{s_{1}}e
    )\rfloor -1)\geq 0\,.
  \end{align*}
  This yields \eqref{eq:8} in the case where $C=\AF^1_\KK$.

  \medskip
  
  Now, we show the converse. Assume that $C=\AF^1_\KK$ and that
  $\partial$ stabilizes $A$. Recall that $\partial$ stabilizes $A$ if
  and only if \eqref{eq:8} holds. If $\omega$ is the unique maximal
  cone in $\Lambda(\DD)$, then $h_{z}$ is identically zero for all
  $z\in C$ and so $(i')$ is trivially satisfied. Therefore the lemma
  follows in this case.
  
  In the sequel, we assume that $\Lambda(\DD)$ has at least two
  maximal cones. Let $\omega_{0}\in \Lambda(\DD)$ be a maximal cone
  different from $\omega$. Then there exists a lattice vector
  $m\in \relint\omega_{0}$ such that $h_{z}(m)\in\ZZ$ and
  $\partial^{(lp^{s_{1}})}(\varphi_{m})\neq 0$ for some
  $l\in\ZZ_{\geq 0}$. Note that here
  $\ker \partial =
  \bigoplus_{m\in\omega_{M}}\KK\cdot\varphi_{m}\chi^{m}$.
  Taking $m$ big enough we may suppose that $-h_{z}(m)\geq lp^{s_{1}}$
  and by Lemma~\ref{sec:line} we may suppose that
  \[
  \ord_{z}\partial_{t}^{(lp^{s_{1}})}(\varphi_{m})=-h_{z}(m) - l.
  \]
  By \eqref{eq:8} we have
  \begin{align}
    \label{eq:9}
    \lfloor h_{z}(m + lp^{s_{1}}e)\rfloor - h_{z}(m) - l \geq 0.
  \end{align}
  Letting $\bar{h}_{z}$ be the linear extension of
  $h_{z}|_{\omega_{0}}$ we have
  \begin{align}
    \label{eq:10}
    \lfloor h_{z}(m + lp^{s_{1}}e)\rfloor = \lfloor h_{z}(m) +
    l\bar{h}_{z}(p^{s_{1}}e)\rfloor = h_{z}(m) + \lfloor
    l\bar{h}_{z}(p^{s_{1}}e)\rfloor\,.
  \end{align}
  Now, \eqref{eq:9} and \eqref{eq:10} yield
  \[
  l\bar{h}_{z}(p^{s_{1}}e)\geq \lfloor l\bar{h}_{z}(p^{s_{1}}e)\rfloor
  \geq l
  \]
  and so $\bar{h}_{z}(p^{s_{1}}e)\geq 1$. Finally, letting
  $m\in\sigma^{\vee}_{M}$, we obtain
  \[
  \lfloor h_{z}(m + p^{s_{1}}e)\rfloor \geq \lfloor h_{z}(m)\rfloor +
  \lfloor \bar{h}_{z}(p^{s_{1}}e)\rfloor\geq \lfloor h_{z}(m)\rfloor +
  1\,.
  \]
  This yields $(i')$ and so concludes the proof of the lemma in the
  case where $C=\AF^1_\KK$.

  \medskip
  
  Assume now that $C=\PP^1_\KK$. Then for $z\in C\setminus\{\infty\}$
  and for any $m\in\sigma^{\vee}_{M}$ such that $A_{m}\neq 0$, we can
  find $\varphi_{m,z}\in A_{m}$ satisfying
  $\ord_{z}(\varphi_{m,z})+\lfloor h_{z}(m)\rfloor=0$.  Replacing
  $\varphi_{m}$ by $\varphi_{m,z}$ in the previous argument and using
  Lemma~\ref{sec:line} for $z=\infty$ in an analog way as in the above
  proof, we obtain the equivalence between \eqref{eq:8} and
  $(i'),(iii')$.

  \medskip
  
  \emph{Case $h$ integral}. Again in this case we have $d=1$. Let
  $v\in N$ be such that $\langle m,v\rangle=h(m)$ for all $m\in
  \omega_M$.  Let us consider the polyhedral divisor defined by $\DD'
  = \DD + (-v+\sigma)\cdot 0$ if $C$ is affine, and by $\DD' = \DD +
  (-v+\sigma)\cdot 0 + (v+\sigma)\cdot\infty$ if $C$ is
  projective. Now $A$ is equivariantly isomorphic to $A[C,\DD']$ and
  $A[C,\DD']$ is as in the case where $h=0$.  Conjugating $\partial$
  by the equivariant isomorphism $A\simeq A[C,\DD']$ (see
  \cite[Proposition~4.5]{La2}), the algebra $A$ is
  $\partial$-invariant if and only if assertions $(i'), (iii')$ hold
  for the polyhedral divisor $\DD'$. An easy computation shows that
  this is equivalent to $\DD$ satisfying $(i),(ii),(iii)$.

  \emph{General case.} Now, we assume that $h$ is not integral, i.e.,
  that $d>1$. Let us consider the normalization $B$ of the cyclic
  extension $A[\zeta^{-dh(w)}\chi^w]\subseteq \KK(\zeta)[M]$, where
  $\zeta^d=t$ and $w\in\relint(\omega)\cap M$ satisfies
  $\gcd(dh(w),d)=1$. We remark that $B$ is naturally
  $M$-graded. Furthermore,
  \[K'_0=\left\{\frac{a}{b}\mid a,b\in B_m,\ m\in M,\mbox{ and }b\neq
    0\right\}=\KK(\zeta)\,.\] %
  Hence, $B=A[C',\DD']$, where $C'\simeq \PP^1_\KK$ if $A$ is elliptic
  and $C'\simeq \AF^1_\KK$ otherwise. We let $k$ and $d'$ be the
  unique pair of positive integers such that $d=d'p^k$ with
  $\gcd(d',p)=1$. Let $\pi:C'\rightarrow C$ be the morphism induced by
  the field inclusion $K_0=\KK(t)\subseteq \KK(\zeta)=K'_0$. Then by
  Lemma~\ref{lm1.9}, Remark~\ref{rm:5.9} and \cite[Section~3.12,
  Exercise~3.8]{St}, we obtain
  \[\DD'=
  \begin{cases}
    d\cdot \Delta_0\cdot[0]+\sum_{z'\in
      C'\setminus\{0\}}p^k\cdot\Delta_{z}\cdot z', & \mbox{if }
    C=\AF^1_\KK \\
    d\cdot \Delta_0\cdot[0]+\sum_{z'\in
      C'\setminus\{0,\infty\}}p^k\cdot\Delta_{z}\cdot z'+d\cdot
    \Delta_\infty\cdot[\infty], & \mbox{if } C=\PP^1_\KK

  \end{cases}
  \]
  
  This yields $h'_0=dh_0$, $h'_\infty=dh_\infty$ and
  $h'_{z'}=p^kh_{z}$, where $\pi(z')=z$ and $h'_{z'}$ is the support
  function of the coefficient $\Delta'_{z'}$ of $\DD'$ at
  $z'$. Moreover, $h'_0|_\omega$ is integral and so the algebra $B$
  satisfies the conditions of the previous case ($h$ integral). We let
  $h':M_\RR\rightarrow \RR$ be the linear extension of $h'_0|_\omega$.

  Let
  \[B_M=\bigoplus_{m\in M}\varphi'_m\cdot\KK[\zeta]\cdot\chi^m,
  \quad\mbox{where}\quad \varphi'_m=\zeta^{-dh(m)}\,.\] %
  Since $A_M\subseteq B_M$ is a cyclic extension, by
  Corollary~\ref{cor:2.6} the LFIHD $\partial$ on $A_M$ extends to an
  LFIHD $\partial'$ on $B_M$. Furthermore, $\partial$ stabilizes $A$
  if and only if $\partial'$ stabilizes $B$ (see the argument in
  \cite[Lemma~3.26]{Li}).

  By the previous case, $B$ is stabilized by $\partial'$ if and only if
  for every $m\in \sigma^\vee_M$ such that $m+p^{s_1}e\in
  \sigma^\vee_M$, the following conditions are satisfied.
  \begin{enumerate}[$(i)$]
  \item[$(i'')$] If $h'_{z'}(m+p^{s_{1}}e)\neq 0$, then $\lfloor
    h'_{z'}(m+p^{s_{1}}e)\rfloor - \lfloor h'_{z'}(m)\rfloor\geq 1$,
    $\forall z'\in C'$, $z'\neq 0,\infty$.
  \item[$(ii'')$] If $h'_{0}(m+p^{s_{1}}e)\neq h'(m + p^{s_{1}}e)$, then $\lfloor
    h'_{0}(m+p^{s_{1}}e)\rfloor - \lfloor h'_{0}(m)\rfloor\geq 1 +
    dh'(p^{s_{1}}e)$.
  \item[$(iii'')$] If $C = \PP^{1}_{\KK}$, then $\lfloor
    h'_{\infty}(m+p^{s_{1}}e)\rfloor - \lfloor h'_{\infty}(m)\rfloor
    \geq -1 - h'(p^{s_{1}}e)$.
  \end{enumerate}
  Now, the lemma follows replacing $h'$ by $dh$, $h'_0$ by $dh_0$,
  $h'_\infty$ by $dh_\infty$ and $h'_z$ by $p^kh_z$ for all $z'\in
  C'$, $z\neq 0,\infty$.
\end{proof}

The following is our main result in this section. It gives a
classification of horizontal LFIHDs on affine $\TT$-varieties of
complexity one over a perfect field. It is a direct consequence of the
results in this section.

\begin{theorem} \label{th:5.11}
  Assume that the base field $\KK$ is perfect.  Let $p=\chara\KK$ if
  $\chara\KK>0$ and $p = 1$ if $\chara\KK=0$. Let $\DD$ be a proper
  $\sigma$-polyhedral divisor over a regular curve $C$ and let
  $A=A[C,\DD]$. Let $\omega\subseteq M_{\RR}$ be a rational cone and let
  $e\in M$ be a lattice vector.

  Then there exists a homogeneous LFIHD on $A$ of horizontal type with
  $\deg\partial = e$ and with $\omega$ as weight cone of
  $\ker\partial$ if and only if the following conditions hold.
  \begin{enumerate}
  \item[$(i)$] $C = \AF^{1}_{\KK}$ or $C = \PP^{1}_{\KK}$.
  \item[$(ii)$] If $C = \AF^{1}_{\KK}$, then $\omega$ is a maximal cone
    in the quasifan $\Lambda(\DD)$, and there exists a rational point
    $z_{0}\in C$ such that $h_{z|\omega}$ is integral $\forall z\in C,
    z\neq z_{0}$.
  \item[$(ii')$] If $C = \PP^{1}_{\KK}$, then there exists a rational
    point $z_{\infty}$ such that $(ii)$ holds for $C_{0}
    :=\PP^{1}_{\KK}\setminus \{z_{\infty}\}$.
  \end{enumerate}
  Without loss of generality, we may suppose that $z_{0} = 0$,
  $z_{\infty} = \infty$, and $h_{z}|_\omega = 0$ $\forall z\in C,
  z\neq 0,\infty$. Let also $h$ be the linear extension of
  $h_{0}|_\omega$ to $M_\RR$ given by $h(m)=\langle m,v\rangle$ for
  some $v\in N_\QQ$, let $d>0$ be the smallest integer such that $dh$
  is integral and let $k$ be the unique non-negative integer such that
  $d=d'p^k$, with $\gcd(d',p)=1$. Let $\tau=\omega^\vee$ and denote by
  $\widehat{\tau}$ the cone in $\widehat{N}_\RR$ generated by $(v,1)$
  and $(\tau,0)$ if $C=\AF_\KK^1$ and by $(v,1)$,  $(\tau,0)$ and
  $(\Delta_\infty,-1)$ if $C=\PP_\KK^1$.
  \begin{enumerate}
  \item [$(iii)$] There exists $s_1\in \ZZ_{\geq 0}$ such that
    $\big(p^{s_1}e,-1/d - h(p^{s_1}e)\big)\in\rt\widehat{\tau}$.
  \end{enumerate}
  For any $m\in\sigma^{\vee}_{M}$ such that
  $m+p^{s_{1}}e\in\sigma^{\vee}_{M}$ the following hold.
  \begin{enumerate}
  \item[$(iv)$] If $h_{z}(m+p^{s_{1}}e)\neq 0$, then $\lfloor
    p^kh_{z}(m+p^{s_{1}}e)\rfloor - \lfloor p^kh_{z}(m)\rfloor\geq 1$,
    $\forall z\in C, z\neq0,\infty$.
  \item[$(v)$] If $h_{0}(m+p^{s_{1}}e)\neq h(m + p^{s_{1}}e)$, then
    $\lfloor dh_{0}(m+p^{s_{1}}e)\rfloor - \lfloor
    dh_{0}(m)\rfloor\geq 1 + dh(p^{s_{1}}e)$.
  \item[$(vi)$] If $C = \PP^{1}_{\KK}$, then $\lfloor
    dh_{\infty}(m+p^{s_{1}}e)\rfloor - \lfloor dh_{\infty}(m)\rfloor
    \geq -1 - dh(p^{s_{1}}e)$.
  \end{enumerate}
  More precisely, all possible homogeneous LFIHD $\partial$ on $A$ of
  horizontal type with $e, \omega$ satisfying $(i)-(iv)$ are given by
  the formula \eqref{eq:7} in Theorem~\ref{th:5.8}. If $\chara\KK>0$,
  then $\partial$ is described by a sequence of integers $0\leq
  s_{1}<s_{2}<\ldots <s_{r},$ where every $\big(p^{s_i}e,-1/d -
  h(p^{s_i}e)\big)$ belongs to $\rt\widehat{\tau}$. Moreover,
  \[
  \ker\partial = \bigoplus_{m\in\omega_{L}}\KK\varphi_{m}\chi^{m},
  \]
  where $L = h^{-1}(\ZZ)$ and $\varphi_{m}\in A_{m}$ satisfies the
  relation
  \[
  \divi\varphi_{m} + \DD(m)=0\quad \mbox{if}\quad C =
  \AF^{1}_{\KK};\qquad \mbox{or}\qquad(\divi\varphi_{m})|_{C_{0}} +
  \DD(m)|_{C_{0}}=0\quad \mbox{if}\quad C = \PP^{1}_{\KK}.
  \]
\end{theorem}

\begin{exemple}
  Let the notation be as in Example~\ref{ex:1.8}. By
  Theorem~\ref{th:5.11}, there exists a homogeneous LFIHD on $A$ with
  degree $\deg\partial=e=(1,2)$ and with weight cone $\omega$ of
  $\ker\partial$ equal to the cone generated by $(0,1)$ and $(1,1)$ in
  $M_\RR$. Indeed, $(i)$ holds since $C=\PP^1_\KK$ and $(ii)'$ holds
  with $z_0=0$ and $z_\infty=\infty$. With this choice,
  $h_z|_\omega=0$ for all $z\in C, z\neq 0,\infty$. The vector $v\in
  N_\RR$ such that $h(m)=\langle m,v\rangle$ corresponds to
  $v=(1/2,0)$.  The cone $\tau$ is generated in $N_\RR$ by $(1,0)$ and
  $(-1,1)$ and the cone $\widehat{\tau}$ in $\widehat{N}_\RR$ is
  generated by $(1,0,2)$, $(-1,1,0)$ and $(1,0,-2)$. Taking $s_1=0$,
  we have that $(e,-1)=(1,2,-1)\in\rt\widehat{\tau}$ so that $(iii)$
  holds. Furthermore, a straightforward verification shows that
  $(iv)$, $(v)$ and $(vi)$ hold.
\end{exemple}

\begin{exemple}
  We assume in this example that the ground field $\KK$ is
  algebraically closed of characteristic $2$. Let us consider the
  Bertin surface
  $$W_{2,5} = \{x^{2}y = x + z^{5}\}\subseteq \AF_{\KK}^{3}$$
  of type $(2,5)$. This is a smooth affine surface endowed with the
  $\G_{m}$-action
  $$\lambda\cdot (x,y,z) = (\lambda^{5}x, \lambda^{-5}y, \lambda z),$$
  where $\lambda\in \G_{m}$ and $(x,y,z)\in W_{2,5}$. Consider the
  polyhedral divisor
  $$\DD = \left\{\frac{1}{5}\right\}\cdot [0] + \left[0, \frac{1}{5}\right]\cdot [1]$$ 
  over the affine line $\AF^{1} = \AF_{\KK}^{1}$. Here we have
  $N = M = \ZZ$.  The elements
  $$x = t^{-1}\chi^{5},\,\, y  = (t+1)t\chi^{-5},\,\, z = \chi^{1}$$
  generate the $\ZZ$-graded algebra $A = A[\AF^{1}, \DD]$ and satisfy
  the equation of $W_{2,5}$. Hence we may identify the
  $\G_{m}$-surface $X = \spec\,A$ with $W_{2,5}$. The quotient map by
  the $\G_{m}$-action is
  $$\pi:(x,y,z)\mapsto t = xy +1.$$  
  The fiber $\pi^{-1}(1)$ consists in two distinct toric curves which
  intersect only at the origin:
  $$\pi^{-1}(1) = \{(0, y, 0)\,|\, y\in \KK\}\cup \{(z^{5}, 0, z)\,|\,z\in\KK\}.$$  
  In the setting of Theorem~\ref{th:5.11}, we may take $z_{0} = 0$ so
  that $\tau = \RR_{\geq 0}$ and
  $$\hat{\tau} = \RR_{\geq 0}(1,0) + \RR_{\geq 0}(1,5).$$

  If $e = 1$ and $s := s_{1} = 2$, then
  $(2^{s}e, -\frac{1}{5} - \frac{2^{s}e}{5}) = (4, -1)$ is a Demazure
  root of $\hat{\tau}$ with distinguished ray $(1,5)$. Condition (iv)
  of Theorem~\ref{th:5.11} is not fulfilled. The corresponding
  homogeneous iterative higher derivation $\partial$ verifies the
  formula
  $$e^{\alpha \partial}(t^{l}\chi^{m}) = \sum_{i=0}^{\infty}\binom{5l + m}{i} t^{l-i}\chi^{m+4i}\alpha^{4i}$$
  for any $(m,l)\in \ZZ^{2}$.  This implies directly that
  $$e^{\alpha \partial}(x) = x \text{ and } e^{\alpha \partial}(z)  = z + \alpha^{4}x,$$
  and so the subalgebra $\KK[x, z]\subseteq A$ is $\partial$-stable.
  However, we have $\partial^{(4)}(y) = t\chi^{-1}\not\in A.$

  Now let us take $e = 1$ and $s = 6$. Then
  $(2^{s}e, -\frac{1}{5} - \frac{2^{s}e}{5}) = (64, -13)$ is a
  Demazure root of $\hat{\tau}$. The conditions of Theorem 5.11 are
  satisfied and the associated LFIHD $\partial'$ has exponential map
  $$e^{\alpha \partial'}(t^{l}\chi^{m}) = \sum_{i=0}^{\infty}\binom{5l + m}{i} t^{l-13i}\chi^{m+64i}\alpha^{64i}.$$
  Therefore
  $$e^{\alpha \partial'}(x) = x, \text{  } e^{\alpha \partial'}(z)  = z + \alpha^{64}x^{13},$$
  and
  $$e^{\alpha\partial'}(y) = x^{-1}(1+e^{\alpha\partial'}(t))= y + \alpha^{64} x^{11} z^{4} + \alpha^{256} x^{50} z + \alpha^{320} x^{63}.$$
  The kernel of $\partial'$ is the subalgebra $\KK[x]\subseteq A$.
\end{exemple}

\begin{remarque}
  A generalization of \cite[Section 4.1]{Li} allows to define and
  compute the homogeneous Makar-Limanov invariant of an affine $\TT$-variety of
  complexity one of arbitrary characteristic. Due to lack of space, we
  omit this straightforward generalization.
\end{remarque}

\section*{Acknowledgements}

We thanks the referee for valuable remarks. We thank the Institut
Fourier, where most of this work was carried out, for its support and
hospitality. The first author also thanks the jury members of his Ph.D
thesis for many suggestions and corrections.

The first author was partially supported by the Max Planck for
Mathematics, Bonn. The second author was partially supported by
Fondecyt project 11121151 and by internal funds from Direcci\'on de
Investigaci\'on, Vicerrector\'\i a Acad\'emica, Universidad de Talca.

\end{document}